\documentclass{amsart}
\usepackage[utf8]{inputenc}
\usepackage{amsmath,amssymb,amsthm, graphics, comment,bm}
\usepackage{mathtools}
\usepackage{amsthm}
\usepackage{diagmac2}
\usepackage{textcomp}
\usepackage{color}
\usepackage[T1]{fontenc}
\usepackage[alphabetic]{amsrefs}
\usepackage[all,cmtip,color,matrix,arrow]{xy}
\usepackage{yfonts}
\usepackage{tikz}
\usepackage{tikz-cd}
\usepackage{stmaryrd}
\usepackage{mathrsfs}
\usepackage{enumerate}
\usepackage{hyperref}
\hypersetup{
	colorlinks=true,
	linkcolor=blue,
	filecolor=magenta,      
	urlcolor=cyan,
}
\usepackage[mathscr]{euscript}
\usepackage[all,cmtip]{xy}
\usepackage{cleveref}

\calclayout

\DeclareSymbolFont{script}{U}{eus}{m}{n}
\DeclareMathSymbol{\Wedge}{0}{script}{"5E}

\newcommand{\cA}{\mathscr{A}} 
 
\newcommand{\cG}{\mathcal{G}} 
\newcommand{\cI}{\mathcal{I}}

\renewcommand{\AA}{\mathbb{A}} 

\newcommand{\NN}{\mathbb{N}} 
 
\newcommand{\CC}{\mathbb{C}} 
\newcommand{\cC}{\mathcal{C}} 
 
\newcommand{\ZZ}{\mathbb{Z}} 
\newcommand{\cX}{\mathcal{X}} 
\newcommand{\cZ}{\mathcal{Z}} 
\newcommand{\cY}{\mathcal{Y}} 
\newcommand{\cE}{\mathcal{E}}

\newcommand{\GGm}{\mathbb{G}_m}
\newcommand{\GG}{\mathbb{G}}
 
\newcommand{\PP}{\mathbb{P}} 
 
\newcommand{\cL}{\mathcal{L}} 
 
\renewcommand{\AA}{\mathbb{A}} 
\newcommand{\QQ}{\mathbb{Q}}
\newcommand{\cF}{\mathcal{F}} 

\newcommand{\cU}{\mathcal{U}}  
\newcommand{\cP}{\mathcal{P}} 

\newcommand{\Proj}{\operatorname{Proj}}

\newcommand{\id}{\operatorname{id}}

\newcommand{\im}{\operatorname{im}}
\newcommand{\Bl}{\operatorname{Bl}}

\newcommand{\Aut}{\operatorname{Aut}}

\newcommand{\HH}{\operatorname{H}}

\newcommand{\cd}{\mathfrak{d}}

\newcommand{\Pic}{\operatorname{Pic}}

\newcommand{\Exc}{\mathrm{Exc}}
\newcommand{\pr}{\mathrm{pr}}
\newcommand{\Hom}{\mathrm{Hom}}

\newcommand{\Spec}{\mathrm{Spec}}

\newcommand{\sProj}{\mathscr{P}{\mathrm{roj}}}

\newcommand{\ext}{\mathrm{ext}}
\newcommand{\codim}{\mathrm{codim}}
\newcommand{\depth}{\mathrm{depth}}

\usepackage{etoolbox}

\newcounter{step}
\AtBeginEnvironment{proof}{\setcounter{step}{0}}

\newcommand{\step}{%
  \refstepcounter{step}%
  \par\smallskip
  \noindent\textbf{Step \thestep.}\quad
}

\newcounter{case}
\AtBeginEnvironment{proof}{\setcounter{case}{0}}

\newcommand{\case}{%
  \refstepcounter{case}%
  \par\smallskip
  \noindent\textbf{Case \thecase.}\quad
}

\let\bar\overline
\let\phi\varphi
\let\tilde\widetilde
\let\hat\widehat
\let\epsilon\varepsilon
\newtheorem{theorem}{Theorem}[section]

\newtheorem*{Teo*}{Theorem}
\newtheorem{Lemma}[theorem]{Lemma}

\newtheorem{Cor}[theorem]{Corollary}
\newtheorem{example}[theorem]{Example}

\newtheorem{Prop}[theorem]{Proposition}
\newtheorem{Question}[theorem]{Question}
\newtheorem*{Ques*}{Question}

\theoremstyle{definition}

\newtheorem*{Oss'}{Remark}
\newtheorem{EG}[theorem]{Example}
\newtheorem{Def}[theorem]{Definition}
\newtheorem*{Def*}{Definition}
\newtheorem{Notation}[theorem]{Notation}
\newtheorem{Remark}[theorem]{Remark}

\numberwithin{equation}{section}

\begin{document}
\title{Characterizing Contractions and Weighted Blowdowns}
\author[S.~Ghosh]{Soham Ghosh}
	\address{University of Washington, Seattle, Washington, USA}
	\email{soham13@uw.edu}
\author[T.~Klingner]{Tyson Klingner}
	\address{University of Washington, Seattle, Washington, USA}
	\email{tysonk4@uw.edu}

\begin{abstract}
This paper gives a partial answer to a question of Dan Abramovich:
consider a proper morphism $f : \cX \to \cZ$ with connected fibers, between smooth separated Deligne--Mumford stacks, which defines an isomorphism away from a smooth effective Cartier divisor $\cE \subseteq \cX$. Then, is $f$ a weighted blowup? We confirm that $f$ is an ordinary smooth blowup when $\cX$ and $\cZ$ are smooth separated schemes of finite type over $\CC$, and when $f :\cX \to \cZ$ is a representable morphism of smooth separated Deligne--Mumford stacks. Further, we show that $f$ is a weighted blowup when $\cX$ and $\cZ$ are smooth separated Deligne--Mumford surfaces, i.e., $\dim \cX = \dim \cZ = 2$. As an application we determine when a reduction morphism between Hassett moduli stacks of weighted stable curves is given by a blowup along a smooth center.
\end{abstract}  

\maketitle

 \section{Introduction and Statement of Result}

 Blowups are paramount to studying the birational geometry of schemes.
 Hironaka's resolution of singularities shows one can produce a desingularization of a given variety over a field of characteristic zero by a sequence of blowups along smooth centers \cite{MR199184}.
 Further, any birational map between smooth varieties $X \dashrightarrow Y$ can be factored into a sequence of blowups and blowdowns along smooth centers \cite{MR1896232}.
 The utility and applications of blowups motivate the following question: 
 given a variety $X$ can we realize it as a blowup of another variety $Z$? 
 Searching for such $Z$ motivates the minimal model program which has birthed a wealth of insight into the birational classification of varieties, see \cites{MR903730, KollarMori1998, MR1875410, MR2601039, MR2675555}. 
 Consequently, it is natural to ask the analogous question for Deligne--Mumford stacks. 
 When dealing with stacks there are two cases to consider: when a morphism is representable or non-representable. 
 Blowups along smooth centers are representable and many results regarding blowups of schemes are known to be true for stacks, e.g., \cite[Lem. 4.11]{kresch2023birationalgeometrydelignemumfordstacks}. 
 The generalization of a blowup to the non-representable case is a {\em weighted blowup} \cite{QR} (see Definition \ref{WB}). 
 Weighted blowups yield more efficient algorithms for resolving singularities than Hironaka's algorithm \cite{MR4856603}.
 Hence, we can ask a natural question for stacks:
 given a stack $\cX$ can we realize it as a weighted blowup of another stack $\cZ$?
 More precisely, this led Dan Abramovich to the following question.

 \begin{Question}
 \label{conje}      
    Suppose $\cX$ and $\cZ$ are smooth separated Deligne--Mumford stacks, and there is a contraction $f : \cX \to \cZ$ where the exceptional $\cE$ is a smooth irreducible Cartier divisor. 
    Moreover, suppose there is a smooth closed irreducible Deligne--Mumford substack $\cY \subseteq \cZ$ such that $(f^{-1}\cY) \simeq \cE$ set-theoretically, and $f : \cX \setminus \cE \xrightarrow{\sim} \cZ \setminus \cY.$    
    Then is $f$ a regular weighted blowup with reduced center $\cY$?
 \end{Question}

 In this article, we affirm {\em two} cases of the question. We give a complete proof of the case when $f : \cX \to \cZ$ is representable, showing that the morphism $f$ is an ordinary smooth blowup with reduced center $\cY.$ 
 As a corollary, we see that the question is true when $X$ and $Z$ are schemes where $f$ is the blowup along a smooth center $Y$, i.e., $X \simeq \Bl_Y Z$. 
 Finally, we prove the case when $\cX$ and $\cZ$ are Deligne--Mumford surfaces. 
  { 
 \begin{theorem}\label{rep theorem}
Suppose $\cX$ and $\cZ$ are smooth separated Deligne--Mumford stacks, and there is a representable contraction morphism $f : \cX \to \cZ$ whose exceptional $\cE$ is a smooth Cartier divisor. Moreover, suppose there is a smooth closed Deligne--Mumford substack $\cY \subseteq \cZ$, such that $f^{-1}(\cY) \simeq \cE$ set-theoretically, and $f : \cX \setminus \cE \xrightarrow{\sim} \cZ \setminus \cY$. Then $f$ is an ordinary blowup with reduced center $\cY$, in particular $f|_{\cE} : \cE \to \cY$ is a projective bundle, and $\mathcal{O}_{\cX}(\cE)|_{\cE} \simeq \mathcal{O}_{\cE}(-1) \otimes f^*\cL$ for some line bundle $\cL$ on $\cY$. 
 \end{theorem}
 }

 \begin{Cor}
 \label{main scheme}
    Suppose $X$ and $Z$ are smooth separated schemes, of finite type over $\mathbb{C}$, and there is a projective contraction $f : X \to Z$ morphism whose exceptional $E$ is {a smooth Cartier divisor}.
    Moreover, suppose there is a {smooth closed subscheme $Y \subseteq Z$}, such that $f^{-1}(Y) \simeq E$ set-theoretically, and $f : X \setminus E \xrightarrow{\sim} Z \setminus Y$.
    Then $f$ is a blowup along smooth center $Y$, i.e., $X \simeq \Bl_{Y} Z$.
    In particular, $f|_{E} : E \to Y$ is a projective bundle, and $\mathcal{O}_{X}(E)|_{E} \simeq \mathcal{O}_{E}(-1)$.
 \end{Cor}
 The complex manifold analogue of Corollary~\ref{main scheme} appears in Russian in \cite{Moishezon66}, with an English translation in \cite{Moishezon67}. The proof realizes $\Bl_Y Z$ as the incidence variety parameterizing complex lines normal to the center $Y$ and identifies the space with the closure of the pullback of these normal lines along the contraction $f:X\rightarrow Z$, which in turn is proven to be isomorphic to $X$ by local analytic computations. Our proof is substantially different as we first identify the projective bundle structure of the exceptional $E\rightarrow Y$ and utilize techniques such as bend-and-break, which came after \cite{Moishezon66}. We believe our approach is useful to the reader interested in Moishezon spaces.
 Note that the assumptions in the statement are essential. 
 For instance, there are birational contraction morphisms between smooth separated schemes of finite type over $\CC$, which are not given by a sequence of blowups along smooth centers, e.g., \cite{406702}.

 The second case of Question~\ref{conje} which we prove is the following.

 {\begin{theorem}\label{mainv2}\label{main}
Suppose $\cX$ and $\cZ$ are smooth separated Deligne--Mumford surfaces, and there is a contraction $f : \cX \to \cZ$ whose exceptional $\cC$ is a smooth stacky curve (not necessarily irreducible). Moreover, suppose that there is a closed reduced substack $\cY\subseteq \cZ$ of dimension 0 such that $f(\cC)=\cY$ set theoretically and $f : \cX \setminus \cC \xrightarrow{\sim} \cZ \setminus \cY$ is an isomorphism. Then $f$ is a regular weighted blowup with reduced center $\cY$, in particular $f|_\cC : \cC \to \cY$ is a weighted projective bundle, and $\mathcal{O}_{\cX}(\cC)|_{\cC} \simeq \mathcal{O}_{\cC}(-1) \otimes f^*\cL$ for some line bundle $\cL$ on $\cY$. 
 \end{theorem}} 
{
As an application of Theorem~\ref{rep theorem}, we prove when a reduction morphism $\rho_{\mathcal{B},\cA}:\overline{\mathcal{M}}_{g,\cA}\rightarrow \overline{\mathcal{M}}_{g,\mathcal{B}}$ between two Hassett spaces \cite{Hassett03} is given by an ordinary blowup along a smooth center. 
 \begin{theorem}\label{Theorem:reductionblowup}
    Let $\rho_{\mathcal{B},\cA}:\overline{\mathcal{M}}_{g,\cA}\rightarrow \overline{\mathcal{M}}_{g,\mathcal{B}}$ be the reduction morphism with respect to weight data $\cA, \mathcal{B}$ such that $E_g(\cA, \mathcal{B})$ is a pairwise intersecting antichain in the Boolean lattice of subsets of $[n]$. Then $\rho_{\mathcal{B}, \cA}$ is an ordinary blowup of $\overline{\mathcal{M}}_{g,\mathcal{B}}$ along a smooth center with exceptional divisor $\cE_g(\cA, \mathcal{B})$.
\end{theorem}
}
\subsection{Strategy of Proof} 

{To prove Theorem~\ref{rep theorem} we first prove Corollary~\ref{main scheme}. 
Indeed,} we prove Corollary~\ref{main scheme} by first demonstrating the case when the center $Y$ of the contraction $f:X\rightarrow Z$ is a point, see Proposition~\ref{contraction to a point case}. This is done by showing that the line bundle $L:=\mathcal{O}_X(-E)|_E$ is ample on the exceptional $E$ and the canonical bundle of $E$ is $K_E=-nL$, using smoothness of $E$ and bend-and-break ($n=\dim X$). By Kobayashi-Ochiai's characterization of projective space \cite{KO73} we obtain $(E, H)\cong (\PP^{n-1}, \mathcal{O}_{\PP^{n-1}}(1))$, whereby the proof follows by Zariski's Main Theorem. 
To prove Corollary~\ref{main scheme} for the general center $Y$ of positive dimension $r$, we reduce to the point case using Bertini's theorem to cut down $Z$. Then to prove Theorem~\ref{rep theorem} we reduce to the case of Corollary~\ref{main scheme} by working \'etale-locally.

To prove Theorem~\ref{main} we show that the contracted smooth stacky curve $\cC$ in the statement of the theorem is necessarily algebraically simply-connected. We then prove that any algebraically simply-connected smooth stacky curve is $\cC \simeq \cP(a,b)$ for some $a,b>0.$ We conclude using the pushout properties of weighted-blowups \cite[Thm. 1.1]{SWB2024}.

 \subsection{Organization of Manuscript} In Section \ref{overview} we recall some preliminaries on the structure of smooth separated Deligne--Mumford stacks with generic trivial stabilizers and the theory of stacky proj and (regular) weighted blowups following \cite{MR3719470} and \cite{QR} respectively. We also note a criterion for smooth weighted blowdowns \cite{SWB2024} which leads to Question~\ref{conje}. In Section~\ref{Section:Scheme case} we prove Corollary~\ref{main scheme}. This is established by first proving the case when the center $\cY$ is a point, see Proposition~\ref{contraction to a point case}, and then reducing the general center $\cY$ to the point case by an involved Bertini argument in Section~\ref{Slicing Section}.

 In Section~\ref{section:representable} we prove Theorem~\ref{rep theorem} by working \'etale-locally and reducing to the case of Corollary~\ref{main scheme}. In Section~\ref{section:DMsurface} we prove Theorem~\ref{main}. First, in Section~\ref{contracted cartier divisor} we prove that the only smooth Cartier divisors within a smooth Deligne--Mumford stack (of any dimension) which can be contracted to a point in a smooth scheme are algebraically simply-connected, see Theorem~\ref{contracted}. 
 In Section~\ref{sc stacky curves}, we prove that the only smooth algebraically simply-connected stacky curves with projective coarse moduli space are weighted projective curves $\cP(a,b).$
 Next, in Section~\ref{neg} we first prove that the contracted smooth stacky curve $\cC\subseteq \cX$, when $\dim \cX=2$, is negative that is, $\mathcal{O}_{\cX}(\cC)|_{\cC} \simeq \mathcal{O}_{\cC}(-m)$ for some $m > 0$. We conclude by showing $m = 1$ and completing the proof of Theorem~\ref{main}.\
  Finally, in Section~\ref{section:applications} we give an application of Theorem~\ref{rep theorem}, demonstrating when a reduction morphism between two Hassett spaces is given by an ordinary blowup along a smooth center.

 \subsection*{Conventions} Throughout this article we are working over the field $\CC$ and all schemes and stacks are assumed to be noetherian and essentially of finite type over $\CC$. 
 A {\em smooth stacky curve} is a smooth, proper Deligne--Mumford stack of dimension 1 with a {\em projective} coarse moduli space. 
 A contraction morphism $f : \cX \to \cZ$ of smooth, separated Deligne--Mumford stacks is a surjective morphism satisfying $f_*\mathcal{O}_{\cX} \simeq \mathcal{O}_{\cZ}$ whose induced morphism on the coarse moduli spaces $f : X \to Z$ is projective.

 \subsection*{Acknowledgements} The authors thank Giovanni Inchiostro, S\'andor Kov\'acs, and Michele Pernice for insightful, thoughtful, and productive conversations regarding birational geometry, weighted blowups, and stack theory.
 Their patience, insight, and wealth of knowledge is indispensable in this project. The authors are also thankful to Dan Abramovich and Andrea Di Lorenzo for valuable comments and suggestions on an earlier draft. The authors also thank our fellow graduate students Ting Gong for enlightening conversations about gerbes, and Jay Reiter for explaining topological invariance of the fundamental group. TK is supported by a Fulbright Future Scholarship funded by the Kinghorn Foundation and supported in part by funds from NSF grants DMS-2502104 and DMS-2100389. SG is supported by a Simons Dissertation Fellowship SFI-MPS-SDF-00014771 and partially by NSF CAREER DMS-2044564 and NSF FRG DMS-2151718 grants.

 \section{Preliminaries}
 \label{overview}

 In this section, we briefly recall a convenient characterization of smooth separated Deligne--Mumford stacks and the theory of weighted blowups, which will be used in the proof of Theorem~\ref{main} in Section~\ref{section:DMsurface}. 
 We refer the reader to \cite{kresch2023birationalgeometrydelignemumfordstacks} for more on the  birational geometry of Deligne--Mumford stacks.

 \subsection{Root stacks and smooth separated Deligne--Mumford stacks with generically trivial stabilizers} Let $\cX$ be an algebraic stack equipped with a line bundle $\cL$ on $\cX$ and section $s \in \Gamma(\cX, \cL)$. To this data, one can associate a stack $\cX(\sqrt[\leftroot{-2}\uproot{2}r]{\cL,s})\rightarrow \cX$ called the $r^{th}$-root stack of $(\cL, s)$ on $\cX$ for all $r>0$, as follows.
 
 \begin{Def}[{\cite[Sec. 10.3]{MR3495343}}]
    The data $(\cL, s)$ defines a morphism $[\cL, s] : \cX \to [\AA^1/\GG_m]$, where $\GG_m$ acts on $\AA^1$ with weight 1.
   For any $r>0$, let $\mathbf{r} : [\AA^1/\GG_m] \to [\AA^1/\GG_m]$ be the morphism induced by $(M, t) \mapsto (M^{\otimes r}, t^{\otimes r})$ on line bundles and sections.
   Then, the {\em r-th root stack $\cX(\sqrt[\leftroot{-2}\uproot{2}r]{\cL,s})$ of $\cX$ along $(\cL,s)$} is the fiber product
   \[
   \begin{tikzcd}
   {\cX(\sqrt[\leftroot{-2}\uproot{2}r]{\cL,s})} \arrow[d] \arrow[r] & {[\AA^1/\GG_m]} \arrow[d, "\mathbf{r}"] \\
   \cX \arrow[r, "{[\cL, s]}"]                                     & {[\AA^1/\GG_m]}               
   \end{tikzcd}
   \]
 \end{Def} 
 
 \begin{Notation}
   If $D$ denotes the Cartier divisor associated to $s \in \Gamma(\cX, \cL)$ then we write $\cX(\sqrt[\leftroot{-2}\uproot{2}r]D) := \cX(\sqrt[\leftroot{-2}\uproot{2}r]{L,s})$.
 \end{Notation}  
 
 \begin{Remark}
 \begin{enumerate}
     \item Root stacks provide examples of algebraic stacks that are birational to the original stack, but have additional stacky structures. Consequently, root stacks are not related to the original stack by a sequence of classical blowups. 
     \item By \cite[Sec 10.3.9, p.~217]{MR3495343}, a morphism $T\rightarrow \cX(\sqrt[\leftroot{-2}\uproot{2}r]{\cL,s})$ is equivalent to the data of a morphism $f:T\to\cX$, a line bundle $\cG$ on $T$, a section $s_{\cG}\in\Gamma(T,\cG)$, and an isomorphism $\phi:\cG^{\otimes r}\xrightarrow{\sim}f^*\cL$ such that $\phi(s_{\cG}^{\otimes r})=f^*s$.
 \end{enumerate}
 \end{Remark}

 Suppose now that $\cX$ is a smooth separated Deligne--Mumford stack with coarse moduli space $\pi : \cX \to X$. If $\cX$ has generic trivial stabilizer, then one can remarkably recover the stack $\cX$ from $X$ and the ramification data.  
 This ``bottom up" characterization is due to the following result of Geraschenko and Satriano.

 \begin{theorem}[{\cite[Thm. 1]{MR3719470}}]\label{Theorem:bottomup}
   Suppose $\cX$ is a smooth separated Deligne--Mumford stack with coarse moduli space $\pi : \cX \to X$ and trivial generic stabilizer. 
   Then, we can recover the stack $\cX$ from $X$ as follows
   \[
   \cX \simeq X^{\mathrm{can}}(\sqrt{(D,\boldsymbol{e})})^{\mathrm{can}} \to X^{\mathrm{can}}(\sqrt{(D,\boldsymbol{e})}) \to X^{\mathrm{can}} \to X,
   \]
   where $(-)^{\mathrm{can}}$ denotes the canonical stack associated to a stack of finite quotient singularities (c.f. \cite[Prop. 2.8]{MR1005008}), and $(D, \boldsymbol{e})$ denotes the ramification data of $\cX \to X^{\mathrm{can}}$.
 \end{theorem}  

\begin{Remark}
    More generally, a smooth separated Deligne--Mumford stack $\cX$ is uniquely determined by its coarse moduli space $X$ and an open smooth Deligne--Mumford substack $\cU \subseteq \cX$ containing all codimension one points, from the uniqueness of \cite[Thm. 3.1]{SWB2024}.
 \end{Remark}

 \begin{Remark}
   If $\cX$ is a smooth separated stacky curve, then the bottom up characterization is simpler. Indeed, the coarse moduli space of a smooth separated stacky curve is smooth, which eliminates the need for the canonical stack. Hence, the bottom up characterization for smooth separated stacky curves only needs root stacks, see Proposition \ref{Bottom up}.
 \end{Remark}

 \subsection{Normalization of Deligne--Mumford stacks}
Many standard constructions for schemes extend to algebraic stacks. 
 For example if $\cA$ is a quasi-coherent sheaf of $\mathcal{O}_{\cX}$-algebras on $\cX$ then there is a {\em relative spectrum} $\underline{\Spec}_{\cX} \, \cA$ \cite[Sec. 10.2.1]{MR3495343} which defines an algebraic stack.
 Using the relative spectrum we can define the {\em normalization} of a reduced Deligne--Mumford stack. In what follows, sheaves over $\cX$ are defined on the canonically associated small \'etale site $\cX_{\text{\'et}}$. 
 First, we require the following proposition

 \begin{Prop}[{\cite[\href{https://stacks.math.columbia.edu/tag/03GG}{Tag 03GG}]{stacks-project}}]
 \label{smooth norm}
Let $Z$ be a scheme, $\mathcal{B}$ a quasi-coherent $\mathcal{O}_Z$-algebra and $\tilde{\mathcal{B}}$ be the integral closure of $\mathcal{O}_Z$ in $\mathcal{B}.$ If $f : X \to Z$ is a smooth morphism, then $f^*\tilde{\mathcal{B}}$ is identified with the integral closure of $\mathcal{O}_X$ in $f^*\mathcal{B}.$
 \end{Prop}

\begin{Def}[Normalization]
Let $\cX$ be a reduced algebraic stack and let $\cA$ be the sheaf of $\mathcal{O}_{\cX}$-algebras defined by setting $\mathcal{A}(\Spec \, A \to \cX)$ to be the integral closure of $A$ in the total ring of fractions $Q(A)$ where $\Spec \, A \to \cX$ is an \'etale affine chart. Since normalization commutes with \'etale extensions (Proposition \ref{smooth norm}) $\cA$ is quasi-coherent, and we define the {\em normalization} $\nu : \cX^\nu \to \cX$ to be 
\[
\cX^\nu := \underline{\Spec}_{\cX} \, \cA.
\]
\end{Def}

\begin{Remark}
\label{reduced not int}
Here we state two remarks.
\begin{enumerate}
\item Similar to schemes, the normalization is unique up to unique isomorphism.

\item If $\cX$ is disconnected and every connected component is integral, then the normalization is precisely the disjoint union of the normalization of each component. That is, by writing $\cX = \cX_1 \cup \cdots \cup \cX_n$ 
\[
\cX^\nu = \bigsqcup_{i=1}^n \cX_i^\nu.
\]
\end{enumerate}
\end{Remark}

Similar to schemes there is also a notion of relative normalization

\begin{Def}[Relative normalization]
Let $f : \cY \to \cX$ be a quasi-compact and quasi-separated morphism of reduced Deligne--Mumford stacks so that $f_* \mathcal{O}_{\cY}$ is a quasi-coherent sheaf of $\mathcal{O}_{\cX}$-algebras. Let $\cA$ denote the normalization of $\mathcal{O}_{\cX}$ in $f_* \mathcal{O}_{\cY}$ (\cite[\href{https://stacks.math.columbia.edu/tag/035F}{Tag 035F}]{stacks-project}). We define the {\em normalization} of $\cX$ in {$f :\cY \to \cX$} to be
\[
\cX^\nu_f := \underline{\Spec}_{\cX} \, \cA.
\]
\end{Def}

\begin{Remark}
Here we state two remarks.
\label{rel norm}
\begin{enumerate}
\item The relative normalization comes equipped with a natural factorization of $f$ (\cite[\href{https://stacks.math.columbia.edu/tag/0BAK}{Tag 0BAK}]{stacks-project})
\[
\cY \xrightarrow{f'} \cX_f^\nu \xrightarrow{\nu} \cX,
\]
and the relative normalization is unique up to unique isomorphism.
\item By unwinding the definitions, we observe that if $\cY$ is disconnected with each connected component integral, and $\cX$ is integral, then the normalization of $\cX$ in $f : \cY \to \cX$ is precisely the disjoint union of the normalization of $\cX$ in each $f_i : \cY_i \to \cX$, i.e.,
\[
\cX^\nu_f = \bigsqcup_{i=1}^n X^\nu_{f_i}
\]
where $\cY = \cY_1 \cup \cdots \cup \cY_n$ are the components.
\end{enumerate}
\end{Remark}

We now state three results, which we will need in Section \ref{section:DMsurface}.
The first two are standard and are presented without proof.

\begin{Lemma}
\label{int closed lemma}
Let $A$ be a noetherian normal ring and suppose $V \subseteq \Spec \,A$ is an open dense subscheme.
Then, $A$ is integrally closed in $\Gamma(V, \mathcal{O}_V)$ under the structure morphism $A \to \Gamma(V, \mathcal{O}_V).$
\end{Lemma}

\begin{Prop}
\label{norm in open subscheme}
Suppose $\cX$ is a normal (potentially disconnected) Deligne--Mumford stack, and suppose $U \subseteq \cX$ is a dense open subscheme. Then, the normalization of $\cX$ in $i : U \hookrightarrow \cX$ is precisely $\cX$.
\end{Prop}

\begin{Prop}
\label{rel norm in open subscheme}
Suppose $f : \cY \to \cX$ is a finite \'etale morphism of reduced Deligne--Mumford stacks where $\cY$ is normal and contains an open dense subscheme $U$. Then, the normalization of $\cX$ in $f|_U : U \to \cX$ is precisely $f :\cY \to \cX$, i.e., $\cX^\nu_{f|_U} \simeq \cY.$
\end{Prop}

\begin{proof}
It suffices to work \'etale-locally.
Let $X \to \cX$ be an \'etale atlas where $X$ is reduced, and set $Y := \cY\times_{\cX} X$. Since finite \'etale is stable under basechange the induced morphism $f : Y \to X$ is finite \'etale. Similarly, by letting $V := U \times_{\cY} Y$ it follows that $V \subseteq Y$ is an open dense subscheme.
Now, let $\Spec \,A \subseteq X$ be a reduced affine scheme.
Since $f : Y \to X$ is finite and $Y$ is normal, $Y \times_{X} \Spec \, A \simeq \Spec \, B$ is a normal affine scheme where $A \to B$ is integral. Consider now $A \to B \hookrightarrow{\Gamma(W, \mathcal{O}_W)}$ where $W := V \cap \Spec \, B.$ Note that the second morphism is injective since $W \subseteq \Spec \, B$ is open dense, and $B$ is reduced.
Suppose that $f \in \Gamma(W, \mathcal{O}_W)$ is integral over $A$, then $f$ is integral over $B$ and by Lemma \ref{int closed lemma} it follows that $f \in B$.
On the other hand, if $f \in B$, then the image $f \in \Gamma(W, \mathcal{O}_W)$ is integral over $A$ since $A \to B$ is integral. Therefore, the integral closure of $A \to \Gamma(W, \mathcal{O}_W)$ is $B$, and hence, $\cX_{f|_U}^\nu \times_{\cX} X \simeq Y$. Thus, $\cX_{f|_U}^\nu \simeq \cY$ and the normalization of $\cX$ in $f|_U : U \to \cX$ is $f :\cY \to \cX.$
\end{proof}

 \subsection{Weighted Embeddings, Stacky Proj, and Weighted Blowups} In this section we briefly review weighted blowups and related theory, following \cite{QR}, and \cite[Sec. 2]{SWB2024}. 
 Recall that in the case of a smooth stack, an ordinary blowup replaces the center with a projective bundle that parameterizes the normal directions to the center in the ambient space. 
 The same holds true for weighted blowups with the added subtlety that each normal direction has a positive integer weight.
 Thus, an ordinary closed embedding is insufficient to define a weighted blowup since one needs to encode a positive integer weight for each normal direction. This necessitates the notion of a {\em weighted embedding}. Similarly the classical proj construction is inadequate and one needs a generalization called the {\em stacky proj}. We recall these notions below.

Let $X$ be a noetherian scheme and $R$ be a quasicoherent graded sheaf of $\mathcal{O}_X$-algebras. 
The grading on $R$ can be described by the coaction 
\[
\beta : R \to R \otimes_{\mathcal{O}_X} \mathcal{O}_X[t, t^{-1}] = R[t, t^{-1}]
\]
defined by  $R_n \ni r \mapsto rt^n \in R_n[t, t^{-1}]$.
Equivalently, the grading is described by an action
\begin{equation}\label{Equation:gradedaction}
\alpha : \GGm \times \underline{\Spec}_X(R) \to \underline{\Spec}_X(R).
\end{equation}
Let $R_+$ denote the ideal of $R$ generated by $\oplus_{n > 0} \,R_n$.
The action \eqref{Equation:gradedaction}  preserves $\underline{\Spec}_X(R)\setminus V(R_+)$, whereby the stack-theoretic Proj is defined as follows.

\begin{Def}[Stack-Theoretic Proj]
For $X$ and $R$ as above, the {\em stack-theoretic proj} of $R$ is defined as the quotient stack
\[
\sProj_X(R) := [(\underline{\Spec}_X(R)\setminus V(R_+))/\GG_m]
\]
\end{Def}

We can realize many standard stacks by the stack theoretic proj.
For instance, we can express a {\em weighted projective stack} in terms of a stack-theoretic proj.
Recall the definition of a weighted projective stack.

\begin{Def}[weighted projective stack and curve]
Consider the $(a_1, \ldots, a_n)$-weighted action of $\GG_m$ on $n$-dimension affine space $\AA^n$, i.e., \[t \cdot (x_1, \ldots, x_n) = (t^{a_1}x_1, \ldots, t^{a_n}x_n)\] where each $a_i > 0.$ 
Then, the {\em weighted projective stack} $\mathcal{P}(a_1, \ldots, a_n)$ is given by the stack quotient
\[
\cP(a_1, \ldots, a_n) := [(\AA^n \setminus 0)/\GG_m].
\]
When $n = 2$, we refer to the weighted projective stack $\mathcal{P}(a,b)$ as a {\em weighted projective curve}.
\end{Def}

The following lemma lists a few properties of the stacky proj.

\begin{Lemma}[{\cite[Proposition 1.6.1]{QR}}]
   Let $R$ be a graded $\mathcal{O}_X$-algebra. Then,
   \begin{enumerate}
      \item[(i)] $\sProj_X(R)$ has a finite diagonal relative to $X$. 
      In particular, $\sProj_X(R)$ is separated over $X$.
      \item[(ii)] If $R$ is finitely generated, then $\sProj_X(R)$ is proper over $X$.
      \item[(iii)] The coarse moduli space of $\sProj_X(R)$, relative to $X$, is the standard relative proj, i.e., $\Proj_X(R).$
   \end{enumerate}   
\end{Lemma}   

The next lemma describes the stabilizer group $G_x$ of a geometric point $x : \Spec \, \CC \to \underline{\Spec}_X(R)$ of the stacky proj. 
Recall that the stabilizer group $G_x$ is a closed group scheme of $\GG_m \times \Spec \, \CC$ defined by the following Cartesian diagram
\[
\begin{tikzcd}
G_x \arrow[r] \arrow[d]                                 & {\Spec \, \CC} \arrow[d, "{(x, x)}"] \\
\GGm \times \underline{\Spec}_X(R) \arrow[r, "{(\alpha, \pi_2)}"] & \underline{\Spec}_X(R) \times_X \underline{\Spec}_X(R)      
\end{tikzcd}
\]

\begin{Lemma}[{\cite[Lem. 1.1.2]{QR}}]
    \label{stack proj stab}
The Cartier dual of $G_x$ is $\ZZ/(d : x \not \in V(R_d)),$ that is,
\begin{itemize}
    \item[(i)] $G_x = \GGm$ if and only if $x \in V(R_+),$ and
    \item[(ii)] $\mu_d \subset G_x$ if and only if $x \in V(R_n)$ for all $n$ such that $d \not | \, n.$
\end{itemize}    
\end{Lemma}

In particular, $V(R_+)$ precisely contains all fixed points of the $\GGm$-action.
Moreover, the $\GG_m$-action is free if $R$ is generated in degree $1$, since then  $(R_n : d \nmid n) = R_+$
for all $d > 1$. In this case the stacky proj $\sProj_X(R)$ is in fact isomorphic to the standard relative proj $\Proj_X(R)$. We now recall the notion of a weighted embedding.

 \begin{Def}{{\cite[Def. 3.1.1 and 4.1.1]{QR}}}
   A {\em weighted embedding} $Y_\bullet \hookrightarrow X$ is defined by a sequence of closed embeddings $\{Y_n = V(I_n) \hookrightarrow X\}_{n \ge 0}$ such that $\bigoplus_{i \ge 0} I_n t^n$ defines a \textit{Rees algebra} that is, a quasicoherent, finitely-generated, graded $\mathcal{O}_X$-subalgebra of $\mathcal{O}_X[t]$, satisfying
   \begin{enumerate}
      \item[(i)] $I_0 \supset I_1 \supset I_2 \supset \cdots$
      \item[(ii)] $I_nI_m \subset I_{n+m}$
      \item[(iii)] Locally in the smooth topology on $X$, there exists a sufficiently large positive integer $d$ such that for each $n \ge 1$,
      \[
       I_n = \left ( I_1^{l_1} I_2^{l_2} \cdots I_d^{l_d} \ : \ l_i \in \NN, \ \sum_{i=1}^d il_i = n \right )
      \]
      in which case we say $I_\bullet$ is generated in degrees $\le d.$ 
   \end{enumerate}  
 \end{Def} 
 
 \begin{Remark}
   \label{associated weighted embedding}
   A standard closed embedding $Y = V(I) \hookrightarrow X$ canonically defines a weighted embedding by setting $I_n := I^n$ for $n \ge 0.$
 \end{Remark} 
 
 \begin{Def}
   The {\em weight} of a local coordinate $x$ is defined as $\max\{n\mid x\in I_n\}$.
 \end{Def}  

 \begin{Def}[Weighted Blowup]
  \label{WB}
   Let $Y_\bullet \hookrightarrow X$ be a weighted embedding defined by ideal sheaves $I_n$ and let $I_\bullet = \bigoplus_{n \ge 0} I_nt^n \subset \mathcal{O}_X[t]$ be the associated Rees algebra.
   Let $I_+$ be the ideal generated by $t$.
   The {\em weighted blowup of $X$ along $Y = V(I_1)$} is defined by the morphism
   \[
   \Bl_{Y}X := \sProj_X(I_\bullet)\rightarrow X.
   \]
   We call $Y$ the {\em reduced center} of the weighted blowup.
 \end{Def} 

 \begin{Remark}
   If $Y \hookrightarrow X$ is a closed embedding and $Y_\bullet \hookrightarrow X$ is the associated canonical weighted embedding from Remark \ref{associated weighted embedding}, then the weighted blowup $\Bl_Y X$ is precisely the ordinary blowup by Lemma \ref{stack proj stab}.
 \end{Remark}  

 \begin{EG}
   \label{2,2}
   Consider the sequence of ideals
   \[
   \CC[x,y] \supset (x,y) \supset (x,y) \supset (x^2, y) \supset (x^2, xy, y^2) \supset (x^3, xy, y^2) \supset \cdots
   \]
   These ideals define a weighted embedding $I_\bullet$ of the origin in $\AA^2.$
   In this case, $x$ has weight $2$, $y$ has weight $3$, and this weighted embedding defines the $(2,3)$-weighted blowup of the origin in $\AA^2$.
 \end{EG}  

 \begin{Remark}[{\cite[Rmk. 3.1.4 and 3.2.4]{QR}}]
  \label{extended rees}
   Let $Y_\bullet \hookrightarrow X$ be a weighted embedding defined by ideal sheaves $I_n$. 
   We can extend the Rees algebra $I_\bullet$ to a $\ZZ$-grading by defining the $\mathcal{O}_X[t^{-1}]$-subalgebra $I_\bullet^{\ext} := \bigoplus_{n \in \ZZ} I_n^{\ext} t^n \subset \mathcal{O}_X[t, t^{-1}]$ where
   \[ I_n^{\ext} :=
   \begin{cases}
      I_n & n \ge 0 \\
      \mathcal{O}_X & n < 0.
   \end{cases}
   \]   
   Then, the weighted blowup of $X$ along $Y$ can equivalently be defined by $\sProj_X(I_\bullet^{\ext}) \to X$. 
 \end{Remark}  

 The extended Rees algebra description of a weighted blowup lends itself better to computing the coarse moduli space of weighted blowups. The following example illustrates this and also provides an explicit computation of a weighted blowup.

 \begin{EG}
  \label{2,3 weighted blowup}
   Consider the $(2,3)$-weighted embedding $I_\bullet$ of the origin $0$ in $\AA^2$ from Example \ref{2,2}. 
   Then, 
   \[
   \Bl_0^{(2,3)}\AA^2 := \sProj_{\AA^2}(I_\bullet^{\ext}) = [(\AA^3 \setminus V(I_+))/\GGm]
   \]
   where $\GG_m$ acts on $\AA^3$ with weight $(2,3,-1)$ corresponding to $(x,y,t)$ respectively. 
   Notice that $V(I_+) = \{(0,0,t) \in \AA^3\}$.
   Thus, $\Bl_0^{(2,3)}\AA^2$ defines an open substack of $[\AA^3 /\GG_m]$ given by
   \[
   \Bl_0^{(2,3)}\AA^2 \simeq [((\AA^2 \setminus 0) \times \AA^1) / \GG_m].
   \]
   Alternatively, unwinding the definition gives
   \[
   \Bl_0^{(2,3)} \AA^2 \simeq \sProj_{\AA^2} \left ( \frac{\CC[x,y][u,v,s]}{(x-s^2u, y-s^3v)}\right )
   \]
   where $u$ has degree $2$, $v$ has degree $3$, and $s$ has degree $-1.$
   In general, given the $(a_1, \ldots, a_n)$-weighted embedding of the origin $0$ in $\AA^n$, the weighted blowup $\Bl_0^{(a_1, \ldots, a_n)}\AA^n$ defines an open substack of $[\AA^n/\GG_m]$ under the $(a_1, \ldots, a_n)$-weighted action given by
   \[
   \Bl_0^{(a_1, \ldots, a_n)} \AA^n \simeq [((\AA^{n}\setminus 0) \times \AA^1 ) /\GG_m]
   \]
   which can be described by
   \[
   \Bl_0^{(a_1, \ldots, a_n)} \simeq \sProj_{\AA^n} \left ( \frac{\CC[x_1, \ldots, x_n][y_1, \ldots, y_n, s]}{(x_1 - s^{a_1}y_1, \ldots, x_n - s^{a_n}y_n)} \right )
   \]
   where $\deg(y_i) = a_i$ and $\deg(s) = -1.$
 \end{EG}  

 With a weighted blowup now defined, we can define the corresponding exceptional divisor in this context.

 \begin{Def}[Exceptional Divisor]
   Suppose $Y_\bullet \hookrightarrow X$ is a weighted embedding defined by $I_\bullet.$ 
   The inclusion $I_{\bullet +1} \subset I_\bullet$ induces an inclusion $\mathcal{O}_{\Bl_Y X}(1) \hookrightarrow \mathcal{O}_{\Bl_Y X}$ of invertible sheaves.
   The {\em exceptional divisor} $\cE$ is the Cartier divisor associated with this inclusion. 
   In particular, $\mathcal{O}_{\Bl_Y X}(-\mathcal{E}) \simeq \mathcal{O}_{\Bl_Y X}(1).$
 \end{Def}  

 The exceptional divisor $\mathcal{E}$ of a weighted blowup behaves similarly to that of an ordinary blowup, with the key difference that the following commutative square
 \[
\begin{tikzcd}
\mathcal{E} \arrow[r, hook] \arrow[d] & \Bl_Y X \arrow[d] \\
Y \arrow[r, hook]                     & X                
\end{tikzcd}
 \] 
 need not be Cartesian. In fact, the above square is Cartesian if and only if the weights of each local coordinate is $1$ and it is an ordinary blowup. While the exceptional divisor $\mathcal{E}$, in general, differs from the fiber product $Y \times_X \Bl_Y X$, they share the same underlying reduced structure. The following example illustrates the above subtleties.

 \begin{EG}
  Consider the $(2,3)$-weighted blowup $\pi : \Bl_0^{(2,3)}\AA^2 \to \AA^2$ of the origin in $\AA^2$ from Example \ref{2,3 weighted blowup}. 
  The exceptional divisor $\mathcal{E}$ precisely corresponds to $V(t)$, and thus,
  \[
  \mathcal{E} \simeq \sProj \, \CC[x,y] = \mathcal{P}(2,3).
  \]
  In particular, the coarse moduli space is $\mathbb{P}(2,3)$, which is reduced. 
  However, the fiber product $F := 0 \times_{\AA^2} \Bl_0 \AA^2$ in the category of algebraic stacks defines a closed substack of $\Bl_0\AA^2$ given by
  \[
  F \simeq \sProj_{\AA^2} \left ( \frac{\CC[u,v,s]}{(s^2u, s^3v)} \right ).
  \]
  The associated coarse moduli space is given by
  \[
  \Proj \left ( \frac{\CC[u,v,s]}{(s^2u, s^3v)} \right ),
  \]
  which is non-reduced, and hence, $F \not \simeq \mathcal{E}.$
  However, notice that $F_{\mathrm{red}} \simeq \mathcal{E}.$
 \end{EG}

 \subsection{Relative ampleness}
 In this section we briefly review the notion of relatively ample and uniformizing line bundles for morphisms of algebraic stacks, following \cite{QR}.

 Let $f: \cX\rightarrow \cZ$ be a morphism of quasi-compact algebraic stacks with finite diagonal. Let $p: \cX\rightarrow \cX_{\mathrm{cs}/\cZ}$ be the relative coarse space and $f_{\mathrm{cs}}: \cX_{\mathrm{cs}/\cZ}\rightarrow \cZ$ be separated. Let $\cL$ be a line bundle on $\cX$. Then for $k\gg0$ divisible enough, the line bundle $\cL^{\otimes k}$ descends to $\cX_{\mathrm{cs}/\cZ}$, that is $p_\ast \cL^{\otimes k}$ is a line bundle on $\cX_{\mathrm{cs}/\cZ}$.

\begin{Def}[{\cite[Def.1.7.3]{QR}}]\label{Def:relative ample}
    In the above situation, we define $\cL$ to be relatively ample with respect to $f$ (or $f$-ample), if $p_\ast\cL^{\otimes k}$ is relatively ample on $\cX_{\mathrm{cs}/\cZ}$ with respect to $f_{\mathrm{cs}}$.
\end{Def}

The next lemma shows that relative ampleness is \'etale local on target.

 \begin{Lemma}\label{Lem:etalelocaltarget}
   Let $f: \cX\rightarrow \cZ$ be a morphism of quasi-compact algebraic stacks with finite diagonal and $\cL$ be a line bundle on $\cX$. Then $\cL$ is $f$-ample if and only if for any surjective \'etale representable morphism $\cZ'\rightarrow \cZ$, the pullback $\cL|_{\cX'}$ is $f'$-ample for $f':\cX':=\cX\times_\cZ\cZ'\rightarrow\cZ'$.
 \end{Lemma}

 \begin{proof}
     Let $f:\cX\xrightarrow{p}\cX_{\mathrm{cs}/\cZ}\xrightarrow{\overline{f}}\cZ$ be the relative coarse space factorization. Let $\cZ'\rightarrow \cZ$ be the surjective \'etale representable morphism. If $\cX':=\cX\times_\cZ \cZ'$, then $f':\cX'\xrightarrow{p'}\cX'_{\mathrm{cs}/\cZ'}:=\cX_{\mathrm{cs}/\cZ}\times_\cZ \cZ'\xrightarrow{\overline{f}'}\cZ'$ is the corresponding relative coarse space factorization since relative coarse space for morphisms of tame Artin stacks commutes with representable pullbacks by \cite[Prop.~3.4]{AOV11}. Now $\cL$ is $f$-ample if and only if for some $k\gg 0$ divisible enough $p_\ast\cL^{\otimes k}$ is $\overline{f}$-ample on $\cX_{\mathrm{cs}/\cZ}$. But that is equivalent to $p'_\ast(\cL|_{\cX'})^{\otimes k}\simeq (p_\ast\cL^{\otimes k})|_{\cX'_{\mathrm{cs}/\cZ'}}$ being $\overline{f}'$-ample by \cite[Lem.~74.13.1]{stacks-project}, since $\cX'_{\mathrm{cs}/\cZ'}\rightarrow \cX_{\mathrm{cs}/\cZ}$ is surjective \'etale. By definition, this is equivalent to the pullback $\cL|_{\cX'}$ being $f'$-ample for $f':\cX':=\cX\times_\cZ\cZ'\rightarrow\cZ'$.
 \end{proof}

\begin{Def}\label{Def:uniformizing}
    Let $f: \cX\rightarrow \cZ$ be a morphism of quasi-compact algebraic stacks with finite diagonal and a line bundle $\cL$ on $\cX$. The line bundle $\cL$ is uniformizing relative to $f$, if the induced map $(f, \cL): \cX\rightarrow \cZ\times \mathcal{B}\GG_m$ is representable.
\end{Def}

Let $f: \cX\rightarrow \cZ$ be a qcqs morphism of algebraic stacks and $\cL$ a line bundle on $\cX$. Then given a morphism $\bigoplus_{n\geq 0} f^*f_*\cL^{\otimes n}
  \longrightarrow
  \bigoplus_{n\geq 0} \cL^{\otimes n}$ one obtains an induced morphism
\begin{equation}\label{Eq:phi}
  \varphi_{\cL}: \cX\rightarrow
  \sProj_\cZ\!\left(
    \bigoplus_{n\geq 0} f_*\cL^{\otimes n}\right)
\end{equation}
via the universal property of stacky proj \cite[Prop.~1.5.1]{QR}, such that $\varphi_{\cL}^*\mathcal{O}(1)=\cL$. The following proposition relates relative ampleness and uniformizing to properties of the map $\varphi_\cL$ defined by \eqref{Eq:phi}.

\begin{Prop}[{\cite[Prop. 1.7.6]{QR}}]\label{Prop:amplephi}
    Let $f: \cX\rightarrow \cZ$ be a morphism of quasi-compact algebraic stacks with finite diagonal and $\cL$ be a line bundle on $\cX$.
    \begin{enumerate}
    \item[(i)] If $\cL$ is $f$-ample, then $f^*f_*\cL^{\otimes N}\rightarrow \cL^{\otimes N}$ is surjective for all sufficiently divisible $N$ and thus induces a morphism $\varphi_\cL$ as in \eqref{Eq:phi}.
  \item[(ii)] If $\cL$ is $f$-ample and uniformizing, the induced morphism $\varphi_{\cL}$ is quasi-compact, schematically dominant, open immersion. This morphism is an isomorphism if $f$ is proper.

  \item[(iii)] Assume there exists a positive integer $N$ such that $f^*f_*\cL^{\otimes N}\rightarrow\cL^{\otimes N}$ is surjective. If the induced morphism $\varphi_{\cL}$ is quasi-affine, then $\cL$ is ample and uniformizing.
\end{enumerate}
\end{Prop}
 
 \subsection{Existence of blowdowns}

Given a variety $X$, the problem of existence of a variety $Z$ such that $X$ is a blow up of $Z$ has been widely studied and in various contexts. For smooth surfaces, Castelnuovo’s well-known theorem states that $X$ is the blow-up of a smooth surface $Z$ if and only if it has a rational $-1$-curve. Similar results for higher dimensional varieties exist due to Kodaira \cite{KodairaNAS} and Grauert \cite{Grauert62} (for contracting $n$-dimensional projective spaces), Moishezon \cite{Moishezon66} (in the complex manifold case), Lascu \cite{Lascu69} (with weaker conditions on singularities), and Artin \cite{Artin70} in the context of algebraic spaces with cohomological hypotheses on formal completions. 

More recently, Arena--Di Lorenzo--Inchiostro--Mathur--Obinna--Pernice \cite{SWB2024} proved the existence of weighted blowdowns in the generality of smooth separated Deligne--Mumford stacks
where the exceptional $\cE$ is a {\em weighted projective bundle} (see \cite[Def. 2.15]{SWB2024} for the definition) with positive dimensional fibers over any smooth and separated Deligne--Mumford stack $\mathcal{Y}$.   

 \begin{theorem}[{\cite[Thm. 1.1]{SWB2024}}]
   \label{existence theorem}
   Let $\cX$ and $\cY$ be smooth, separated Deligne--Mumford stacks, and let $\pi : \cE \to \cY$ be a weighted projective bundle with positive dimensional fibers.
   Assume that there is a closed embedding $\cE \hookrightarrow \cX$ with $\cE$ a Cartier divisor in $\cX$, such that $\mathcal{O}_{\cX}(\cE)|_{\cE} \simeq \mathcal{O}_{\cE}(-1) \otimes \pi^* \cL$ for a line bundle $\cL$ on $\cY$.
   Then there is a smooth, separated Deligne--Mumford stack $\cZ$, with two maps $i : \cY \to \cZ$ and $p : \cX \to \cZ$ such that $i$ is a closed embedding and $p$ is a (regular) weighted blowup with reduced center $\cY$.
   Moreover, the resulting square is a pushout in algebraic stacks.
 \end{theorem}

 \subsection{Reduction Steps}

 In this section, we refine and fix our assumptions on the objects present in Theorems~\ref{rep theorem} and \ref{mainv2} to simplify the proof for clarity of exposition.
Indeed, to prove Theorems~\ref{rep theorem} and \ref{mainv2}, it suffices to do so one connected component of $\cX$ at a time (since $f$ induces a bijection between the connected components of $\cX$ and $\cZ$). One can further restrict to the components of $\cX$ which intersect the exceptional $\cE$ non-trivially. Hence, in the rest of the paper we will assume $\cX$ and $\cZ$ are connected. Moreover, since $\cE$ and $\cY$ are smooth, the irreducible components are disconnected, i.e., $\cE=\bigsqcup_{i=1}^{m}\cE_i$ and $\cY=\bigsqcup_{i=1}^{m}\cY_i$.
Further, $f^{-1}(\cY_i)=\cE_i$ set-theoretically, since $f$ has connected fibers. 
Thus, by setting $\cX_i:=\cX\setminus\bigsqcup_{j\neq i}\cE_j$ and $\cZ_i:=\cZ\setminus\bigsqcup_{j\neq i}\cY_j$, we obtain contraction morphisms $f_i:=f|_{\cX_i}:\cX_i\rightarrow \cZ_i$, where $\cX_i$ and $\cZ_i$ are connected, and the exceptional $\cE_i$ is a smooth irreducible Cartier divisor, such that there is a smooth irreducible closed substack $\cY_i\subset \cZ_i$ with $f_i^{-1}(\cY_i)=\cE_i$ set-theoretically. 
Consequently, it suffices to prove the case $\cX$ and $\cZ$ are connected, and $\cE$ and $\cY$ are irreducible. In the case of Theorem~\ref{mainv2} this implies that $f(\cE)=p\in |\cZ|$ and $\cY=\mathcal{B}\Aut_\cZ(p)$.

 \section{General Smooth Scheme Case}\label{Section:Scheme case}

 In this section, we confirm Question~\ref{conje} for the general scheme case.

 \begin{theorem}
 \label{scheme case}
    Suppose $X$ and $Z$ are smooth connected separated schemes of finite type over $\mathbb{C}$ of dimension $n$, and there is a projective contraction morphism $f : X \to Z$ where the exceptional divisor $E$ is a smooth irreducible Cartier divisor. Moreover, suppose there is a smooth closed irreducible subscheme $Y \subseteq Z$ such that $E \simeq f^{-1}(Y)$ set-theoretically, and $f : X \setminus E \xrightarrow{\sim} Z \setminus Y$. 
    
    Then, $f$ is the blowup of $Z$ along smooth center $Y$, i.e., $X \simeq \Bl_Y Z$. In particular, there is a vector bundle $V \to Y$ of rank $r := \operatorname{codim}(Y, Z)$ such that $(E, \mathcal{O}_{E}(-E)) \simeq (\PP(V), \mathcal{O}_{\PP(V)}(1))$ where $\PP(V)$ is the projectivization of $V$.
 \end{theorem}

 \subsection{Contraction to a point}

 To prove Theorem \ref{scheme case} we will use a Bertini style argument to reduce to the case the center is a point.
 Hence, we first prove the case that $Y=p$.

 \begin{Prop}
 \label{contraction to a point case}
    Suppose $f : X \to Z$ is a projective contraction morphism between two smooth connected separated schemes of finite type over $\mathbb{C}$ of dimension $n$ where the exceptional divisor $E$ is smooth and irreducible, and $f$ contracts $E$ to a point $p \in Z.$
    Suppose further that $f : X \setminus E \xrightarrow{\sim} Z \setminus p$.
    Then, $f$ is the blowup of $Z$ at $p$, i.e., $X\simeq \Bl_p Z.$ In particular,
    $E \simeq \PP^{n-1}$, and $\mathcal{O}_X(-E)|_E \simeq \mathcal{O}_E(1)$.
 \end{Prop}

 \begin{Remark}
    Since blowups are local Proposition \ref{contraction to a point case} extends to the case when exceptional $E = \sum_{i =1}^k E_i$ where $f(E_i) = p_i$ are distinct points, i.e., $X \simeq \Bl_{p_1,\ldots, p_k} Z.$
 \end{Remark}
 
 We may assume without loss of generality that  $Z = \Spec(A)$ is a smooth affine scheme where $K_Z \simeq \mathcal{O}_Z$.
 To prove Proposition \ref{contraction to a point case} we need the following theorem of Kobayashi--Ochiai characterizing projective space.

 \begin{theorem}[{\cite[Cor. 1]{KO73}}]
 \label{KOc}
    Suppose $M$ is a compact complex manifold of dimension $m$ with ample line bundle $L$ on $M$. 
    If $c_1(M) \ge (m+1)c_1(L),$
    then $M$ is biholomorphic to $\PP^m$ with $L \simeq \mathcal{O}_M(1).$
 \end{theorem}

 We can now give the proof of Proposition \ref{contraction to a point case}.

 \begin{proof}[Proof of Proposition \ref{contraction to a point case}]
    We will first apply Theorem \ref{KOc} to show that $E \simeq \PP^{n-1}$.
    By \cite[Lem. 2.62]{KollarMori1998} the line bundle $L:= \mathcal{O}_{X}(-E)|_E$ is ample on $E$.
    Since $f : X \to Z$ defines an isomorphism away from $E$ and $f(E)=p$ it follows that $K_X \sim f^*K_Z + aE$ for some $a \in \ZZ$, and since  $K_Z \simeq \mathcal{O}_Z$ we see $K_X \sim aE$.
    In fact, the smoothness forces $a \ge n-1$.
    Indeed, by Artin approximation we may compute this using a manifold chart. 
    Letting $(z_1, \ldots, z_n)$ be a chart of $Z$ centered around $p$ and letting $(x_1, \ldots, x_n)$ be a chart of $X$ where $E$ is defined by $x_1 = 0$ we may write $z_i = x_1^{a_i}u_i(x_1, \ldots, x_n)$ where $a_i \ge 1$ and $u_i(x_1, \ldots, x_n)$ is a non-vanishing analytic function for $i = 1, \ldots, n$. 
    From this, we compute that
    \[
    f^*(dz_1 \wedge \cdots \wedge dz_n) = x^{a_1 + \cdots + a_n -1}u(x_1, \ldots, x_n) \,dx_1 \wedge \cdots \wedge dx_n
    \]
    for some regular analytic function $u(x_1, \ldots, x_n).$
    Here, $a \ge a_1 + \cdots + a_n -1 \ge n-1$ as claimed.
    Then, by the adjunction formula $K_E \sim (a+1)H$ where $H = E|_E$.
    Thus, $-K_E$ is ample, i.e., $E$ is Fano. 
    Hence, by bend-and-break (c.f. \cite[Ch. V Thm. 1.6]{MR1440180}) there is a rational curve $C \subset E$ such that $-K_E \cdot C \le n$ from which it follows that $a+1 \le n$.
    Combining everything, $a=n-1$ and
    $
    c_1(E) = nc_1(L)
    $ so by Theorem \ref{KOc} we see $(E, L) \simeq (\PP^{n-1}, \mathcal{O}_{\PP^{n-1}}(1))$.
    
    To conclude $f : X \to Z$ is the smooth blowup at $p$ we will utilize Theorem \ref{existence theorem}.
    Indeed, Theorem \ref{existence theorem} ensures there is a smooth separated Deligne--Mumford stack $\cZ'$ of dimension $n$ and a morphism $f' : X \to \cZ'$ such that $f'$ is a weighted blowup with reduced center $p$ where $E = \Exc(f')$. 
        Since $X\setminus E \simeq \cZ'\setminus p$ the only possible stacky structure can exist at $p$.
    Moreover, since $E \simeq \PP^{n-1}$ is a scheme it follows that $f'$ is an ordinary blowup, i.e., of weight $(1,\ldots, 1).$
    From the universal property in Theorem \ref{existence theorem} we obtain a unique morphism $g : \cZ' \to Z$, and since $\cZ'$ is a smooth, separated Deligne--Mumford stack there is a coarse moduli space $\pi : \cZ' \to Z'$ and there is an induced morphism $h : Z' \to Z$ giving the diagram
    \[
    \begin{tikzcd}
E \arrow[r, hook] \arrow[d]                    & X' \arrow[rdd, "f", bend left] \arrow[d, "f'"] &   \\
p \arrow[r, hook] \arrow[rd, hook, bend right] & \cZ' \arrow[rd, "{\exists ! \,g}"] \arrow[d, "\pi"]   &   \\
                                               & Z' \arrow[r, "h"]                              & Z
\end{tikzcd}
    \]
It is straightforward to see that $h$ is bijective and birational.
The morphism $h : Z' \to Z$ is finite type, separated, and since $f = h \circ (\pi \circ f')$ is a proper morphism it follows that $h$ is proper \cite[\href{https://stacks.math.columbia.edu/tag/03GN}{Tag 03GN}]{stacks-project}.
By Zariski's main theorem $h$ is an isomorphism.
In other words, $Z$ is the coarse moduli space of $\cZ'.$
However, since $\cZ' \setminus p \simeq Z \setminus p$ we see that $\cZ'$ is a scheme in codimension $n-1 \ge 1$.
Thus, by the bottom up characterization (Theorem \ref{Theorem:bottomup}) it follows that $\cZ' \simeq Z.$
Therefore, $X \simeq \Bl_pZ.$
 \end{proof}

 \subsection{Slicing and Bertini}\label{Slicing Section} 

Let $f : X \to Z$ be the contraction morphism from Theorem \ref{scheme case}, i.e., the exceptional $E$ is smooth irreducible and there is a smooth irreducible subscheme $Y \subseteq Z$ such that $f^{-1}(Y) \simeq E$ as sets and $f : X \setminus E \xrightarrow{\sim} Z \setminus Y$.
We will prove that $X$ is the blowup of $Z$ along the smooth center $Y$, i.e., $X \simeq \Bl_YZ.$
Set $c := \operatorname{codim}_Z(Y)$.
Proposition \ref{contraction to a point case} handles the case $c = n$.
Assume now that $c < n$.
First, we will prove that 
\begin{equation}
\label{fiber}
(f^{-1}(y), \mathcal{O}_X(-E)|_{f^{-1}(y)}) \simeq (\PP^{c-1}, \mathcal{O}_{\PP^{c-1}}(1)), \ \ \text{for general $y \in Y.$}
\end{equation}
To do so, we will employ a Bertini style argument to reduce the smooth center $Y$ to a finite collection of reduced points.
In fact, since we are studying the fiber over $y \in Y$ we may replace $Z$ by an open affine neighborhood of $y \in Y \subset Z$ where we may assume that $Z$ is quasiprojective.
Let $\bar{Z}$ denote the projective closure of $Z$ and fix a very ample line bundle $\cL$ on $\bar{Z}$ defining a closed embedding $\bar{Z} \hookrightarrow \PP\HH^0(\bar{Z}, \cL)$. 
Note that the projective closure $\bar{Y}$ of $Y$ defines a closed subscheme of $\bar{Z}.$
We denote the pullback of $\cL$ to $X$ under the composition $X \xrightarrow{f}Z \hookrightarrow \bar{Z}$ by $f^* \cL,$ which is a basepoint-free line bundle on $X$.
The linear systems which we will apply Bertini theorems to will be restrictions and pullbacks of linear systems of appropriate powers of $\cL$ characterized by the following lemma.

\begin{Lemma}
\label{tech cond}
Suppose $f : X \to Z$ is a birational contraction morphism between two smooth separated schemes where $Z$ is quasi-projective. 
Suppose further that there is a smooth closed subscheme $Y \subseteq Z$ where $f$ defines an isomorphism away from $Y$. 
Let $\bar{Z}$ and $\bar{Y}$ be projective schemes containing $Z$ and $Y$ as schematically-dense open subschemes respectively where $\bar{Y} \subseteq \bar{Z}$ is closed.
Let $\cL$ be a very ample line bundle on $\bar{Z}$, and let $f^*\cL$ denote the pullback of the line bundle under the composition $X \xrightarrow{f} Z \hookrightarrow \bar{Z}.$

Then,
there exists $d \gg 0$ such that for $L := \cL^{\otimes d}$ the finite-dimensional vector space $V := \HH^0(\bar{Z}, L)$ satisfies the following properties.
\begin{itemize}
    \item[(a)] The restriction map $\HH^0(\bar{Z}, L) \to \HH^0(\bar{Y}, L|_{\bar{Y}})$ is surjective.
    \item[(b)] If $V_Y := \im(\HH^0(\bar{Z}, L) \to \HH^0(Y, L|_Y))$, then $V_Y$ is naturally identified with $\HH^0(\bar{Y}, L|_{\bar{Y}}).$ 
    In particular, for $d \gg 0$, $\dim V_Y > r:= \dim Y.$
    \item[(c)] The sections of $V := \HH^0(\bar{Z}, L)$ generate $L|_Z$, and the image sections $V_Y$ generate $L|_Y$.
    Consequently, the pulled-back sections generate $f^*L$ on $X$ and generate $f^* L |_E$ on $E$.
\end{itemize}
\end{Lemma}

\begin{proof}
Let $\cI_{\bar{Y} / \bar{Z}} \subseteq \mathcal{O}_{\bar{Z}}$ denote the ideal sheaf of $\bar{Y}$ in $\bar{Z}.$ 
Consider the short exact sequence
\[
0 \to \mathcal{I}_{\bar{Y}/\bar{Z}} \otimes \cL^{\otimes d} \to \mathcal{L}^{\otimes d} \to \mathcal{L}|_{\bar{Y}}^{\otimes d} \to 0
\]
for $d > 0.$
Passing to the long exact sequence in cohomology gives
\[
\HH^0(\bar{Z}, \cL^{\otimes d}) \to \HH^0(\bar{Y}, \mathcal{L}|_{\bar{Y}}^{\otimes d}) \to \HH^1(\bar{Z}, \mathcal{I}_{\bar{Y}/\bar{Z}} \otimes \cL^{\otimes d}) \to \cdots
\]
Since $\bar{Z}$ is projective and $\cI_{\bar{Y}/\bar{Z}}$ is coherent we can choose $d \gg 0$ such that $\HH^1(\bar{Z}, \mathcal{I}_{\bar{Y}/\bar{Z}} \otimes \cL^{\otimes d}) = 0$ (c.f. \cite[\href{https://stacks.math.columbia.edu/tag/0B5T}{Tag 0B5T}]{stacks-project}). 
Thus, the sequence shows $\HH^0(\bar{Z}, \cL^{\otimes d}) \to \HH^0(\bar{Y}, \cL|_{\bar{Y}}^{\otimes d})$ is surjective, which proves (a).

Since $j : Y \hookrightarrow \bar{Y}$ is a dense open subscheme the structure morphism $\mathcal{O}_{\bar{Y}} \to j_*\mathcal{O}_Y$ is injective, and hence, the induced morphism $\cL|_{\bar{Y}}^{\otimes d} \to j_*(\cL|_{Y}^{\otimes d})$ is injective. 
Since taking global sections is left exact the map $\HH^0(\bar{Y}, \cL|_{\bar{Y}}^{\otimes d}) \to \HH^0(Y, \cL|_Y^{\otimes d})$ is injective. 
Then, combining injectivity with the surjectivity from (a) identifies $V_Y := \im(\HH^0(\bar{Z}, \cL^{\otimes d}) \to \HH^0(Y, \cL|_Y^{\otimes d}))$ with $\HH^0(\bar{Y}, \cL|_{\bar{Y}}^{\otimes d}).$
Then, since $\bar{Y}$ is projective of dimension $r := n-c$ it follows from the asymptotic Riemann-Roch theorem \cite[\href{https://stacks.math.columbia.edu/tag/0BJ8}{Tag 0BJ8}]{stacks-project} that 
\[
\dim V_Y = \dim \HH^0(\bar{Y}, \cL|_{\bar{Y}}^{\otimes d}) \sim \kappa d^r > r \ (\text{for some constant }\kappa)
\]
for $d \gg 0.$
This proves (b).

Finally, for $d\gg0$ the line bundle $L := \cL^{d}$ is very ample, hence globally generated, i.e., $V \otimes_\CC \mathcal{O}_{\bar{Z}} \to L$ is surjective where $V := \HH^0(\bar{Z}, L).$
Since surjectivity is a local property restricting to $Z$ gives a surjection $V \otimes_\CC \mathcal{O}_{Z} \to L|_Z,$ and thus, the global sections of $L \to \bar{Z}$ generate $L|_{Z}.$
Further restricting to $Y$ yields a surjection $V \otimes_\CC \mathcal{O}_{Y} \to L|_Y.$
Using the fact this map factors through $V_Y \otimes_\CC \mathcal{O}_{Y}$ it follows that $V_Y$ generates $L|_Y.$
Moreover, since pullback of sheaves is right exact the pulled-back sections generate $f^*L$ and $f^*L|_E.$
\end{proof}

For convenience we will state the variant of Bertini's Theorem we need to prove Theorem \ref{scheme case}.

\begin{theorem}[{\cite[Cor. 5]{MR360616}}]
Let $W$ be an integral algebraic scheme over an algebraically closed field of characteristic zero, and
let $\mathfrak{d}$ be a finite-dimensional linear system on $W$.
Then, a general member of $\mathfrak{d}$ is regular away from the base locus of $\mathfrak{d}$ and the singular locus of $W$.
In particular, if $W$ is smooth and $\mathfrak{d}$ is basepoint-free, then a general member of $\mathfrak{d}$ is smooth. 
\end{theorem}

\begin{Cor}[{\cite[Rem. 6]{MR360616}}]
\label{bertini cor}
Let $W$ be a smooth scheme over an algebraically closed field of characteristic zero, and let $\cd$ be a finite-dimensional basepoint-free linear system on $W$. 
Then for every $1 \le j \le \dim \cd$ there exists a non-empty Zariski-open subset $\Omega_j \subseteq \cd^{\oplus j}$ such that for $([s_1], [s_2], \ldots, [s_j]) \in \Omega_j$ the zero locus $Z(s_1, \ldots, s_j) \subseteq W$ is either empty or smooth of codimension $j$.
\end{Cor}

Now, fix $d \gg 0$ such that $L := \cL^{\otimes d}$ satisfies the conditions in Lemma \ref{tech cond} where $\cL \to \bar{Z}$ is a very ample line bundle on $\bar{Z}$. Set $V := \HH^0(\bar{Z}, L)$ and let $\cd := \PP(V)$ be the finite-dimensional linear system associated to $L.$ 
We use the convention $\cd^j := \cd \times \cd \times \cdots \times \cd.$

\begin{Prop}
\label{Bertini prop}
Consider the situation of Theorem \ref{scheme case} where we assume $r:= \dim Y > 0.$
Then, there exists a Zariski-open subset $\Omega \subseteq \cd^r$ of ordered $r$-tuples $([s_1], \ldots, [s_r])$ for which
\[
Z_i := V_Z(s_1, \ldots, s_i), \ \ \ X_i = V_X(f^*s_1, \ldots, f^*s_i) = X \times_{Z} Z_i,
\]
\[
Y_i := V_Y(s_1|_Y, \ldots, s_i|_Y), \ \ \ E_i := V_E(f^*s_1|_E, \ldots, f^*s_i|_E) = E\times_Z Z_i \subseteq X_i,
\]
satisfies the following, for all $1 \le i \le r$
\begin{itemize}
\item[(i)] $X_i, Z_i, Y_i$, and $E_i$ are smooth of dimension $n-i, n-i, r-i, n-1-i$, respectively.
\item[(ii)] $Y_r$ is a finite non-empty collection of reduced points
\item[(iii)] If $f_i := f|_{X_i} : X_i \to Z_i$ is the restriction of $f$, then $f_i(E_i) = Y_i$.
\item[(iv)] $E_i$ is a smooth Cartier divisor on $X_i$ where $\mathcal{O}_{X_i}(-E_i) \simeq \mathcal{O}_X(-E)|_{X_i}$.
\item[(v)] The morphism $f_i : X_i \to Z_i$ is projective, birational, and has connected fibers, i.e., $f_i$ is a contraction morphism with exceptional divisor $\Exc(f_i)= E_i$ where $f_i : X_i \setminus E_i \xrightarrow{\sim} Z_i \setminus Y_i$
\end{itemize}
\end{Prop}

\begin{proof}
Set $L_X:=f^\ast L|_Z$, $L_E:=L_X|_E$ and $L_W:L|_W$ for $W\in \{Z, Y\}$. For $W\in\{X, Y, Z, E\}$, let $\rho_W:V\rightarrow H^0(W, L_W)$ be the respective restriction/ pullback maps and define $V_W :=\im\rho_W$. The linear systems associated to $V_X, V_Z$, and $V_Y$ are basepoint-free by Lemma \ref{tech cond} (c).
The linear system $V_E$ is also basepoint-free. 
Indeed, for a given $e \in E$ let $y := f(e)$ and choose a section $s \in V_Y$ such that $s(y)$ does not vanish at $y$.
Then, $f^*s|_E \in V_E$ does not vanish at $e$.
We now claim that $\dim V_W > r = \dim Y$ for every $W \in \{X,Z,Y,E\}.$
By Lemma \ref{tech cond} (b), $\dim V_Y > r$.
The restriction map $V_Z \to V_Y$ is surjective since the surjection $V \to V_Y$ factors through $V \to V_Z$, and hence, $\dim V_Z \ge \dim V_Y > r.$ 
Next, the natural map $f^* : V_Z \to V_X$ is injective.
To see this suppose $f^* t = 0$.
Since $f$ induces an isomorphism $X \setminus E \simeq Z \setminus Y$ it follows that $t|_{Z\setminus Y} = 0.$
Hence, it follows that $t = 0$ since $Z \setminus Y \subseteq Z$ is a dense open subset and $Z$ is smooth, which proves injectivity.
Thus, $\dim V_X \ge \dim V_Z > r.$
Similarly, since $f|_E : E \to Y$ is surjective, $f^* : V_Y \to V_E$ is injective, and thus, $\dim V_E > r.$

Now, by applying Corollary \ref{bertini cor} to each basepoint-free linear system $\PP(V_W)$ for $W \in \{X,Z,Y,E\}$ and for each $1\leq i\leq r$ (note $r<\dim V_W$) we obtain non-empty Zariski-open subsets $\Omega_{W, i} \subseteq \PP(V_W)^{i}$ such that each $([s_1], \ldots, [s_i]) \in \Omega_{W, i}$ corresponds to either the empty scheme or a smooth complete intersection of dimension $\dim W - i$.
Next, each surjective restriction morphism $\rho_W : V \to V_W$ induces a map $\pi_W : \cd \setminus \PP(\ker \rho_W) \to \PP(V_W)$. Let $\pr_i:\cd^r\longrightarrow\cd^i$ denote projection to the first $i$ factors. Let $U_{W, i} := \pr_i^{-1}\left((\pi^i_W)^{-1}\Omega_{W, i}\right) \subseteq \cd^r$ denote the preimage of each Zariski-open subset to the linear system $\cd^{r}$ for each $W \in \{X, Z, Y, E\}.$ Here $(\pi_W^i)^{-1}(\Omega_{W,i})$ is regarded as an open subset of $\cd^i$ contained in
$(\cd\setminus\PP(\ker\rho_W))^i$. For $W\in\{X,Y,Z,E\}$, since $\rho_W$ is surjective, every $U_{W,i}$ is a non-empty Zariski-open
dense subset of $\cd^r$. As $\cd^r$ is irreducible, we obtain Zariski-open dense subset
\[
\Omega_{\mathrm{sm}} := \bigcap_{i=1}^{r}U_{X,i} \cap \bigcap_{i=1}^{r}U_{Z, i} \cap \bigcap_{i=1}^{r}U_{Y, i} \cap \bigcap_{i=1}^{r}U_{E, i} \subseteq \cd^{r}
\]
such that by construction, for every $([s_1], \ldots, [s_r]) \in \Omega_{\mathrm{sm}}$, all partial intersections
\[
Z_i := V_Z(s_1, \ldots, s_i), \ \ \ X_i = V_X(f^*s_1, \ldots, f^*s_i),
\]
\[
Y_i := V_Y(s_1|_Y, \ldots, s_i|_Y), \ \ \ E_i := V_E(f^*s_1|_E, \ldots, f^*s_i|_E),
\]
are smooth of the expected dimension for $1 \le i \le r$ whenever they are non-empty.

To satisfy (i) and (ii) we need to ensure the above intersections are non-empty, and it suffices to show $Y_r\neq \emptyset$ since then all the $Z_i,\ Y_i, X_i,$ and $E_i$ are non-empty.
Further, to satisfy (iii) we also need to refine $\Omega_{\mathrm{sm}}$ to ensure the sections do not mutually vanish along the boundary $B := \bar{Y}\setminus Y$.
To that end, consider $B$ with the induced subscheme structure.
Since $Y \subset \bar{Y}$ is dense every irreducible component of $B$ has dimension at most $r-1.$
Consider the incidence scheme
\[
J_B := \{(b, [s_1], \ldots, [s_r]) \subseteq B \times \cd^r \, | \, s_1(b) = \cdots = s_r(b) = 0\}.
\]
For every $b \in B$ the condition $s_j(b)=0$ is a codimension one condition on $\cd$ since $L$ is globally generated on $\bar{Z}.$
Therefore, \[\dim J_B \le \dim B + r(\dim \cd - 1) < r\dim \cd = \dim \cd^r.\]
The projection $\pr : J_B \to \cd^r$ is proper since $J_B \subseteq B \times \cd^r$ is closed and $B$ is projective. 
Hence, $\pr(J_B) \subseteq \cd^r$ is a proper closed subset, and thus, $\Omega_B := \cd^r \setminus \pr(J_B)$ is a Zariski-open dense subset such that for every tuple $([s_1], \ldots, [s_r]) \in \Omega_B$ the intersection $V_{\bar{Z}}(s_1, \ldots, s_r) \cap \bar{Y}$ has no support along the boundary $B$. 
Finally, since $\bar{Y} \subseteq \bar{Z} \hookrightarrow \cd$ is projective and $V_{\bar{Z}}(s_1, \ldots, s_r)$ is a complementary dimensional linear subspace in $\cd$, there exists a dense open subset $\Omega_{\mathrm{non}} \subseteq \cd^r$ such that for all $([s_1], \ldots, [s_r]) \in \Omega_{\mathrm{non}}$ the intersection $\bar{Y} \cap V_{\bar{Z}}(s_1, \ldots, s_r)$ is non-empty.
Combining everything the Zariski-open subset
\[
\Omega := \Omega_{\mathrm{sm}} \cap \Omega_B \cap \Omega_{\mathrm{non}} \subseteq \cd^r
\]
satisfies (i), (ii), and (iii).
Further, $E_i = V_E(f^* s_1 |_E, \ldots, f^* s_i|_E) = X_i \times_X E$ is smooth for every tuple $([s_1], \ldots, [s_r]) \in \Omega$ by construction.

Lastly, we will show that $\Omega$ satisfies (iv) and (v).
To see that $\Omega$ satisfies (iv) notice that $E_i$ is smooth of dimension $n-1-i$ and $X_i$ is smooth of dimension $n-i$, and moreover $E_i$ has pure codimension 1 in $X_i.$
Further, $E_i = E \times_X X_i$ is the pullback of $E$ along the closed immersion $X_i \hookrightarrow X,$ and
hence, $E_i \subseteq X_i$ is a Cartier divisor.
Consequently, $\mathcal{O}_{X_i}(-E_i) \simeq \mathcal{X}(-E)|_{X_i}$ which proves (iv).
Finally, $f_i : X_i \to Z_i$ is projective by virtue of being the pullback of the projective morphism along $f : X \to Z$ along $Z_i \hookrightarrow Z.$ 
Moreover, $f_i$ induces an isomorphism $X_i \setminus E_i \simeq Z_i \setminus Y_i$ since $f$ induces $X\setminus E \simeq Z\setminus Y$ and $E_i = E\times_X X_i,$ i.e., $f_i$ is birational.
Also, for every $z \in Z_i$ the fibers $f^{-1}(z)$ and $f_i^{-1}(z)$ are isomorphic, thus the fibers of $f_i : X_i  \to Z_i$ are connected.
Therefore, $\Omega$ satisfies condition (v).
\end{proof}

\begin{Prop}
\label{general proj bundle}
In the situation of Theorem \ref{scheme case}, there exists a dense open subset $U \subseteq Y$ such that the restriction $f_U : E_U \to U$ of $f$ to $E_U := f^{-1}(U)$ is a projective bundle, i.e., $E_U : =\PP_U(F)$ for some rank $c := \operatorname{codim}_Z(Y)$ vector bundle $F \to U.$
Moreover, \[(E_U, \mathcal{O}_{E_U}(-E_U)) \simeq (\PP_U(F), \mathcal{O}_{\PP_U(F)}(1) \otimes (f|_U)^*M)\] for some line bundle $M$ on $U$.
Consequently, \[\tilde{X} := X \setminus f^{-1}(Y\setminus U) \simeq \Bl_U(Z \setminus(Y\setminus U)).\]
\end{Prop}

\begin{proof}
We use the same notation as Lemma \ref{tech cond} and Proposition \ref{Bertini prop}.
Consider the incidence scheme
\[
I := \{(y, [s_1], \ldots, [s_r]) \in Y \times \cd^r \, | \, s_1(y) = \cdots = s_r(y) = 0\}.
\]
Let $q : I \to Y$ and $p : I \to \cd^r$ denote the projection maps. 
By Lemma \ref{tech cond} (c) the evaluation morphism $\mathrm{ev} : V \otimes_\CC \mathcal{O}_Y \to L|_Y$ is surjective.
Hence, we have a short exact sequence
\[
0 \to K \to V \otimes_\CC \mathcal{O}_Y \xrightarrow{\mathrm{ev}} L|_Y \to 0
\]
where $K := \ker(\mathrm{ev})$ is locally free of rank $\dim V - 1.$
Now, consider
\[
I_1 := \{(y, [s]) \in Y \times \cd \, | \,s(y) = 0\}
\]
so that $I \simeq I_1^r = I_1 \times_Y I_1 \times_Y \cdots \times_{Y} I_1.$
We claim that the natural projection $I_1 \to Y$ defines a projective bundle where $I_1 \simeq \PP_Y(K).$
Indeed, for a test scheme $T$, a $T$ point of $Y \times \cd$ corresponds to a morphism $g : T \to Y$ and a line subbundle $\ell \subset V \otimes_\CC \mathcal{O}_T$. 
The data $(g, \ell)$ defines a $T$-point of $I_1$ if and only if $\ell$ defines a line subbundle of $K_T := \ker(V \otimes_\CC \mathcal{O}_T \to g^* L|_Y)$.
Hence, $I_1$ defines the same functor as $\PP_Y(K)$, which shows $I_1 \simeq \PP_Y(K).$ 
Consequently, the incidence scheme can be written as $I \simeq \PP_Y(K) \times_Y \cdots \times_Y \PP_Y(K)$, which shows that $q : I \to Y$ is the $r$-fold product of projective bundles over $Y$.
Thus, $q$ is smooth and surjective.
Moreover, since $Y$ is irreducible it follows that $I$ is irreducible.

Now, we will show that $p : I \to \cd^r$ is a dominant morphism. 
Indeed, let $\Omega \subseteq \cd^r$ be the dense open subset constructed in Proposition \ref{Bertini prop}.
Let $([s_1],\ldots, [s_r]) \in \Omega$ be given.
The scheme $Y_r := V_Y(s_1|_Y, \ldots, s_r|_Y)$ is a finite collection of reduced points by Proposition \ref{Bertini prop} (ii).
Hence, for $y \in Y_r$ we have $(y, [s_1], \ldots, [y_r]) \in I$ and its image under $p$ is $([s_1],\ldots, [s_r])$.
Therefore, $\Omega \subseteq p(I)$, which implies that $p$ is dominant. 
Next, consider the Zariski-open dense subset $I_\Omega := p^{-1}(\Omega) \subseteq I$, and let $U := q(I_\Omega) \subseteq Y$, which is open and dense since $q$ is smooth and $Y$ is irreducible. 
For simplicity of notation set $\tilde{X} := X \setminus f^{-1}(Y \setminus U)$ and $\tilde{Z} := Z \setminus(Y \setminus U).$

To finish the proof of the proposition we will now show that $f_U : E_U \to U$ is a projective bundle over $U$ with the desired properties. 
Indeed, for every $y \in U$ there exists sections $s_1, \ldots, s_r \in V$ depending on $y$ such that $s_1(y) = \cdots = s_r(y) = 0$, i.e., $([s_1], \ldots, [s_r]) \in \Omega$ by definition of $U$. 
Hence, letting 
\[
Z_{r} := V_Z(s_1, \ldots, s_r), \ \ \ X_{r} = V_X(f^*s_1, \ldots, f^*s_r),
\]
\[
Y_{r} := V_Y(s_1|_Y, \ldots, s_r|_Y), \ \ \ E_{r} := V_E(f^*s_1|_E, \ldots, f^*s_r|_E),
\]
we see by Proposition \ref{Bertini prop} that $f_r := f|_{X_{r}} : X_{r} \to Z_{r}$ is a contraction morphism between two smooth connected separated schemes of finite type over $\mathbb{C}$ of dimension $n-r$ with $\Exc(f_r) \simeq E_{r}$ and $X_{r} \setminus E_{r} \simeq Z_{r} \setminus Y_r.$
The scheme $Y_r \subseteq Y$ is a finite collection of reduced points $Y_r = \{y_1, \ldots, y_k\}$ containing the given $y \in U.$
Therefore, by Proposition \ref{contraction to a point case} the map $f_r : X_r \to Z_r$ is the blowup of $Z_r$ at the points $y_1, \ldots, y_k$, and $E_r \simeq E_r^1 \sqcup \cdots \sqcup E_r^k$ where $f(E_r^j) = y_j$ and $E_r^j \simeq \PP^{c-1}$ where $c = \codim_Z(Y)$.
Moreover, $\mathcal{O}_{X_r}(-E^j_r)|_{E^j_r} \simeq \mathcal{O}_{\PP^{c-1}}(1)$ for every $j = 1, \ldots, k.$
In particular, since the fibers of $f$ and $f_r$ are isomorphic we have shown that $f^{-1}(y) \simeq \PP^{c-1}$ and $\mathcal{O}_{X}(-E)|_{f^{-1}(y)} \simeq \mathcal{O}_{\PP^{c-1}}(1).$
By hypothesis, $E \subseteq X$ is a smooth Cartier divisor and $Y \subseteq Z$ is smooth with $f(E) = Y$ set theoretically. 
The open dense subset $U \subseteq Y$ is also smooth from which it follows that $E_U := f^{-1}(U)$ is smooth.
Thus, $f_U := f|_{E_U}$ is flat by miracle flatness since every fiber is isomorphic to $\PP^{c-1}.$
Now, $f_U : E_U \to U$ is a Brauer-Severi scheme by \cite[Cor. 8.3]{BSI}.
Further, the line bundle $\mathcal{O}_{\tilde{X}}(-E_U)|_{E_U}$ satisfies $\mathcal{O}_{\tilde{X}}(-E_U)|_{f^{-1}(y)} \simeq \mathcal{O}_{\PP^{c-1}}(1)$.
Indeed, by cohomology and basechange $F := (f_U)_* \mathcal{O}_{E_U}(-E_U)$ defines a rank $c$ vector bundle on $U$ and $f_U : E_U \to U$ is the projective bundle $\PP_U(F)$ on $U$ (i.e., the Brauer-Severi scheme $f_U : E_U \to U$ is trivial, see \cite[Sec. 0]{BSII}).
Moreover, 
\[
\mathcal{O}_{\tilde{X}}(-E_U)|_{E_U} \simeq \mathcal{O}_{\PP_U(F)}(1) \otimes_{\mathcal{O}_{E_U}} (f_U)^* M
\]
for some line bundle $M$ on $U$.
Thus, by Theorem \ref{existence theorem} there is a smooth separated Deligne--Mumford stack $\cZ'$ of dimension $n$ and a morphism $f' : \tilde{X} \to \cZ'$ such that $f'$ is a weighted blowup with reduced center $U$ where $E_U = \Exc(f')$. 
    Since $\tilde{X}\setminus E_U \simeq \cZ'\setminus U$ the only possible stacky structure can exist at $U$.
    Moreover, since $E_U \simeq \PP_U(F)$ is a scheme it follows that $f'$ is an ordinary blowup, i.e., with weight $(1, \ldots, 1).$
    From the universal property in Theorem \ref{existence theorem} we obtain a unique morphism $g : \cZ' \to \tilde{Z}$, and since $\cZ'$ is a smooth, separated Deligne--Mumford stack there is a coarse moduli space $\pi : \cZ' \to Z'$ and an induced morphism $h : Z' \to \tilde{Z}$ giving the diagram
    \[
    \begin{tikzcd}
E_U \arrow[r, hook] \arrow[d]                         & \tilde{X} \arrow[d, "f'"] \arrow[rdd, "\tilde{f}", bend left] &           \\
U \arrow[r, "\iota'", hook] \arrow[rd, "\iota", hook] & \cZ' \arrow[rd, "{\exists ! \, g}"] \arrow[d, "\pi"]          &           \\
                                                      & Z' \arrow[r, "h"]                                             & \tilde{Z}
\end{tikzcd}
\]
    It is straightforward to see that $h$ is bijective and birational.
The morphism $h : Z' \to \tilde{Z}$ is finite type, separated, and since $\tilde{f} = h \circ (\pi \circ f')$ is a proper morphism it follows that $h$ is proper \cite[\href{https://stacks.math.columbia.edu/tag/03GN}{Tag 03GN}]{stacks-project}.
By Zariski's main theorem $h$ is an isomorphism.
In other words, $\tilde{Z}$ is the coarse moduli space of $\cZ'.$
However, since $\cZ' \setminus U \simeq \tilde{Z}\setminus U$ it follows that $\cZ'$ is a scheme in codimension $c -1 \ge 1$.
Thus, by the bottom up characterization (Theorem \ref{Theorem:bottomup}) we have $\cZ' \simeq \tilde{Z}$. Therefore, $\tilde{X} \simeq\Bl_U \tilde{Z}.$
\end{proof}

Lastly, we need the following commutative algebra lemma.

\begin{Lemma}
\label{depth comp} Let $i:Y\hookrightarrow Z$ be a regular closed immersion of positive codimension between locally Noetherian schemes. Let $w\in Y$ and assume $Z$ is regular near $w$ and
$\dim\mathcal{O}_{Y,w}\ge1$. Then, for every $m\ge 0$,
\[
   \depth_{\mathcal{O}_{Z,w}}
   \bigl(\cI_{Y/Z}^{m}\bigr)_w\ge 2.
\]
Thus, in the situation of Theorem~\ref{scheme case} the ideal sheaf $\cI_{Y/Z}$ of $Y \subseteq Z$ satisfies $\depth_{\mathcal{O}_{Z,w}}(\cI_{Y/Z}^m)_w \ge 2$ for every $w \in W:= Y\setminus U$ where $U \subseteq Y$ is the open dense subset from Proposition~\ref{general proj bundle}.
\end{Lemma}

\begin{proof}
For brevity set $A := \mathcal{O}_{Z, w}$, $I^m := (\cI_{Y/Z}^m)_w$, $c := \codim_Z(Y),$ and $e := \dim(A/I) = \dim \mathcal{O}_{Y, w}\geq 1$ (by assumption). Since $Z$ is regular near $w$, $A$ is a regular local ring.
Moreover, the closed immersion $Y \hookrightarrow Z$ is regular, so $I$ is generated by an $A$-regular sequence of length $c$ and thus, $A/I$ is a local complete intersection ring and hence Cohen-Macaulay.
Note that for $m=0$ we have $I^0 = A$ and $\depth_A A  = \dim A = e + c \ge 2$ so the statement is true.

Consider now $m \ge 1$. Thus, $\dim A = e + c \ge 2$ since $c\ge 1.$ By Rees' theorem for regular sequences the natural map
\begin{align*}
(A/I)[T_1, \ldots, T_c] &\to \mathrm{gr}_I(A) = \bigoplus_{r \ge 0} I^r/I^{r+1} \\
T_i &\mapsto x_i \mod I^2
\end{align*}
defines an isomorphism \cite[Thm. 1.1.8]{MR1251956}.
Consequently, every graded piece $I^r/I^{r+1}$ is a finite free $A/I$-module. Since the ring $A/I$ is Cohen-Macauley of dimension $e$, each $A$-module $I^r/I^{r+1}$ has depth $e$.
For $m \ge 1$ the $A$-module $A/I^m$ has a finite filtration
\[
A/I^m \supset I/I^m \supset \cdots \supset I^{m-1}/I^m \supset 0
\]
whose successive quotients are $I^r/I^{r+1}$ for $0 \le r \le m-1.$
By a repeated application of the depth lemma \cite[\href{https://stacks.math.columbia.edu/tag/00LX}{Tag 00LX}]{stacks-project} it follows that $\depth_A(A/I^m) \ge e.$
Finally, applying depth lemma to
\[
0 \to I_m \to A \to A/I^m \to 0
\]
gives
\[
\depth_A(I^m) \ge \min\{\depth_A A, \depth_A(A/I^m) + 1\} \ge \min\{e+c, e+1\} \ge 2.
\]

In the situation of Theorem~\ref{scheme case}, {the open subset $U\subseteq Y$ given by Proposition~\ref{general proj bundle}} is dense in $Y$ and thus, the closed subset $W = Y\setminus U \subseteq Y$ does not contain the generic point of $Y$, implying $e:=\dim\mathcal{O}_{Y,w} \ge 1$. Thus, all the hypotheses of the Lemma are satisfied, implying $\depth_{\mathcal{O}_{Z,w}}(\cI_{Y/Z}^m)_w \ge 2$ for every $w \in W:= Y\setminus U$.
\end{proof}

We can now give the proof of Theorem \ref{scheme case}.

\begin{proof}[Proof of Theorem \ref{scheme case}]
Let $U \subseteq Y$ be the open dense subset from Proposition \ref{general proj bundle}, and let $W := Y\setminus U$, let $V := Z \setminus W$ and let $j : V \hookrightarrow Z$ denote the inclusion. 
Proposition \ref{general proj bundle} states that $X_V := f^{-1}(V) \simeq \Bl_U V.$
We claim that this implies that $X \simeq \Bl_Y Z$ such that the restriction of the isomorphism to $X_V \subseteq X$ is $X_V \simeq \Bl_U V.$
Let $f_V :  X_V \to V$ be the restriction of $f : X \to Z$ and let $E_V := E \times_X X_V.$
The conormal bundle $\mathcal{O}_X(-E)$ is $f$-ample (c.f. \cite[Lem. 2.62]{KollarMori1998}), and hence, we can write
$
X \simeq \underline{\Proj}_Z \bigoplus_{m \ge 0} f_* \mathcal{O}_{X}(-mE).
$
Let $\mathcal{I}_{Y/Z} \subseteq \mathcal{O}_Z$ denote the ideal sheaf of $Y$ in $Z$. 
To prove $X \simeq \Bl_Y Z = \underline{\Proj}_Z \bigoplus_{m \ge 0} \cI_{Y/Z}^m$ it is enough to show 
\begin{align}
\label{hartog}
f_*\mathcal{O}_X(-mE) \simeq \mathcal{I}_{Y/Z}^m
\end{align}
for every $m \ge 0.$
Since $X_V \simeq \Bl_U V$ with exceptional divisor $E_V$ we have
\begin{align}
\label{push forward ideal}
(f_V)_* \mathcal{O}_{X_V}(-mE_V) \simeq \mathcal{I}_{U/V}^m
\end{align}
where $\mathcal{I}_{U/V} \subseteq \mathcal{O}_V$ is the ideal sheaf of $U$ in $V$.
We will extend the isomorphism (\ref{push forward ideal}) over $W$ using Hartog's theorem to obtain (\ref{hartog}). 

First, $j_* \cI_{U/V}^m = \cI_{Y/Z}^m$ for every $m \ge 0$. 
Indeed, since $\cI_{U/V}^m = \cI^m_{Y/Z}|_V$ it follows from \cite[\href{https://stacks.math.columbia.edu/tag/0E9I}{Tag 0E9I}]{stacks-project}
that $j_* \cI_{U/V}^m = \cI_{Y/Z}^m$ provided $\depth_{\mathcal{O}_{Z,w}}(\cI_{Y/Z}^m)_w \ge 2$ for every $w \in W,$ which is true by Lemma~\ref{depth comp}. 

Next, set $\cF_m := f_*\mathcal{O}_X(-mE)$ where $m \ge 0.$
We will now show that $\cF_m \simeq j_*(\cI^m_{U/V})$ to finish the proof.
To do so, we will first show $i_*\mathcal{O}_{X_V}(-mE_V) \simeq \mathcal{O}_X(-mE)$ for every $m \ge 0$ where $i : X_V \hookrightarrow X$ is the open immersion.
Indeed, since $\mathcal{O}_X(-mE)$ is invertible and $\mathcal{O}_X(-mE)|_{X_V} \simeq \mathcal{O}_{X_V}(-mE_V)$ it is enough to show $X_V$ is big.
Recall that $f^{-1}(Y) = E$ set-theoretically, and $\codim_X(f^{-1}(Y)) = 1$.
Further, by Proposition \ref{general proj bundle} we have $f^{-1}(U) \simeq \PP_U(F)$ where $F\to U$ is a vector bundle of rank $c := \codim_Z(Y)$.
Moreover, $\dim f^{-1}(U) = \dim E$ and $f^{-1}(U) \subseteq f^{-1}(Y)$ is a dense open subset where $\codim_{f^{-1}(Y)} f^{-1}(W) \ge 1.$ 
Hence, \[\codim_X f^{-1}(W) \ge 2\] and thus \[i_*\mathcal{O}_{X_V}(-mE_V) \simeq \mathcal{O}_X(-mE)\text{ since }f^{-1}(W) = X \setminus X_V.\]
Finally, pushing forward $i_*\mathcal{O}_{X_V}(-mE_V) \simeq \mathcal{O}_X(-mE)$ along $f$ and using $f \circ i = j \circ f_V$ gives
\[
\cF_m = f_*\mathcal{O}_X(-mE) \simeq f_*i_* \mathcal{O}_{X_V}(-m E_V) 
= j_*(f_V)_* \mathcal{O}_{X_V}(-mE_V).
\]
By (\ref{push forward ideal}) we have $(f_V)_* \mathcal{O}_{X_V}(-mE_V) \simeq \mathcal{I}^m_{U/V}$, hence $\cF_m \simeq j_* \mathcal{I}^m_{U/V}$. 
Combining with $j_* \mathcal{I}^m_{U/V} \simeq \mathcal{I}^m_{Y/Z}$ gives $\cF_m \simeq \cI^m_{Y/Z}$ for every $m \ge 0$ which completes the proof.
\end{proof}

\section{Representable contraction case}\label{section:representable}

 In this section, we confirm Question~\ref{conje} for the case when the contraction $f:\cX\rightarrow\cZ$ is representable.

 \begin{theorem}
 \label{main representable}
    Suppose $\cX$ and $\cZ$ are smooth separated Deligne--Mumford stacks of dimension $n$, and there is a representable contraction $f : \cX \to \cZ$ where the exceptional $\cE$ is a smooth irreducible Cartier divisor. 
    Moreover, suppose there is a smooth closed irreducible Deligne--Mumford substack $\cY \subseteq \cZ$ such that $(f^{-1}\cY) \simeq \cE$ set-theoretically, and $f : \cX \setminus \cE \xrightarrow{\sim} \cZ \setminus \cY.$    
    Then $f$ is an ordinary blowup of $\cZ$ along smooth center $\cY$. In particular, $f|_{\cE} : \cE \to \cY$ is a projective bundle, and $\mathcal{O}_{\cX}(\cE)|_{\cE} \simeq \mathcal{O}_{\cE}(-1) \otimes f^*\cL$ for some line bundle $\cL$ on $\cY$. 
 \end{theorem}

 We first record the following lemma, which is well-known to experts.

 \begin{Lemma}\label{projective represent}
   Let $f:\cX\rightarrow \cZ$ be a representable morphism of separated Deligne-Mumford stacks with respective coarse spaces $\pi_\cX:\cX\rightarrow X$ and $\pi_\cZ:\cZ\rightarrow Z$. If the induced map $\overline{f}:X\rightarrow Z$ is projective, then so is $f:\cX\rightarrow \cZ$. Furthermore, $f$ is representable by schemes.
 \end{Lemma}

\begin{proof}
    Let $\cP:=X\times_Z \cZ$ and consider the map $\gamma:=(\pi_\cX, f):\cX\rightarrow\cP$ . We claim that $\gamma$ is finite. First note that the kernel of the induced map of $\gamma$ on stabilizers $\Aut_\cX(x)\rightarrow \Aut_\cP(\gamma(x))$ is contained in the kernel of the induced map of $f$ on stabilizers $\Aut_\cX(s)\rightarrow \Aut_\cZ(f(x))$, which is trivial by representability of $f$. Thus, $\gamma: \cX\rightarrow \cP$ is representable as well. The projection $\pr_X:\cP\rightarrow X$ is proper as it is the base change of the proper coarse moduli space morphism $\pi_\cZ:\cZ\rightarrow Z$. Furthermore, $\pr_X\circ\gamma=\pi_\cX$, where $\pi_\cX$ is proper and $\pr_X$ is proper and hence separated. Thus, it follows that $\gamma:\cX\rightarrow \cP$ is proper. Since the composite $\pr_X\circ\gamma=\pi_\cX$ is locally quasi-finite as $\pi_\cX$ is coarse space map of a separated Deligne-Mumford stack, it follows that $\gamma$ is locally quasi-finite by representability and \cite[\href{https://stacks.math.columbia.edu/tag/03XN}{Tag 03XN}]{stacks-project}. Thus, $\gamma:\cX\rightarrow \cP$ is representable, proper and locally quasi-finite, which implies that it is finite by \cite[\href{https://stacks.math.columbia.edu/tag/0A4X}{Tag 0A4X}]{stacks-project}, since finiteness of morphism is fpqc local on the base \cite[\href{https://stacks.math.columbia.edu/tag/0426}{Tag 0426}]{stacks-project}. The projection $\pr_\cZ:\cP\rightarrow \cZ$ is projective being the base change of the projective map $\overline{f}:X\rightarrow Z$. Hence $f:\cX\xrightarrow{\gamma}\cP\xrightarrow{\pr_\cZ}\cZ$ is projective, being the composition of two projective maps.
    
    Finally, since $f:\cX\rightarrow \cZ$ is projective and representable, for any map from a scheme $U\rightarrow \cZ$, the base-change $\cX\times_\cZ U$ is an algebraic space, projective over the scheme $U$, and thus, is a scheme itself. This proves that $f:\cX\rightarrow \cZ$ is representable by schemes.
\end{proof}

\begin{Prop}\label{ideal equality}
    Suppose $\cX$ and $\cZ$ are smooth separated Deligne--Mumford stacks of dimension $n$, and there is a representable contraction $f : \cX \to \cZ$ where the exceptional $\cE$ is a smooth irreducible Cartier divisor. 
    Moreover, suppose there is a smooth closed irreducible Deligne--Mumford substack $\cY \subseteq \cZ$ such that $(f^{-1}\cY) \simeq \cE$ set-theoretically, and $f : \cX \setminus \cE \xrightarrow{\sim} \cZ \setminus \cY.$ Then $f^{-1}\cI_{\cY/\cZ}.\mathcal{O}_\cX\simeq\mathcal{O}_\cX(-\cE)$, as ideal sheaves on $\cX$. Here $f^{-1}\cI_{\cY/\cZ}\mathcal{O}_\cX:=\im(f^\ast\cI_{\cY/\cZ}\rightarrow \mathcal{O}_\cX)$. Furthermore, $\mathcal{O}_\cX(-\cE)$ is relatively ample with respect to $f:\cX\rightarrow \cZ$.
\end{Prop}

\begin{proof}
 We fix a smooth surjective \'etale presentation $u:U\rightarrow \cZ$ and let $\cX_U:=\cX\times_\cZ U$, $\cY_U:=\cY\times_\cZ U$ and $\cE_U:=\cE\times_\cZ U$ with induced morphism $f_U: \cX_U\rightarrow U$. By Lemma~\ref{projective represent}, $\cX_U$ is a smooth scheme and thus, $\cX_U\rightarrow \cX$ is a smooth surjective \'etale presentation of $\cX$. Since an isomorphism of ideal sheaves is \'etale local, it suffices to prove $f^{-1}\cI_{\cY/\cZ}.\mathcal{O}_\cX\simeq\mathcal{O}_\cX(-\cE)$ after pulling back to the chart $\cX_U$.
 
 Being base change of closed immersions, the maps $\cY_U\hookrightarrow U$ and $\cE_U\hookrightarrow \cX_U$ are closed immersions, whereby $\cY_U$ and $\cE_U$ are smooth schemes as well and furthermore, $\cE_U\subseteq \cX_U$ is an effective Cartier divisor. Since the restriction $f:\cX\setminus\cE\xrightarrow{\sim} \cZ\setminus \cY$ is an isomorphism, it follows that so is $f_U: \cX_U\setminus \cE_U\xrightarrow{\sim} U\setminus \cY_U$. Since $\cE=f^{-1}(\cY)$ set-theoretically, it follows that $\cE_U=f_U^{-1}(\cY_U)$ set-theoretically as well. Furthermore, since $f$ is projective and $f_\ast\mathcal{O}_\cX\simeq\mathcal{O}_\cZ$ and $u:U\rightarrow \cZ$ is \'etale, and therefore flat, we have $(f_U)_\ast\mathcal{O}_{\cX_U}\simeq u^\ast(f_\ast\mathcal{O}_\cX)\simeq \mathcal{O}_U$ by flat base change for algebraic stacks \cite[\href{https://stacks.math.columbia.edu/tag/0765}{Tag 0765}]{stacks-project}. Since $f_U:\cX_U\rightarrow U$ is projective as well, it follows that $f_U$ has connected geometric fibers. Thus, $f_U:\cX_U\rightarrow U$ satisfies all the hypotheses of Theorem~\ref{scheme case}, except possibly the irreducibility of $\cE_U$ and $\cY_U$.

  To fix this, we refine the chart $u:U\rightarrow \cZ$. First note that restricting to the open $V:=U\setminus \cY_U\subseteq U$, we have $\cX_V:=\cX\times _\cZ V=\cX_U\setminus \cE_U$, whereby $f_V:\cX_V\xrightarrow{\sim} V$ is an isomorphism. Hence, since $\cY_V=\cE_V=\emptyset$,
  \begin{align*}(f^{-1}\cI_{\cY/\cZ}.\mathcal{O}_\cX)|_{\cX_V}\simeq f_V^{-1}\cI_{\cY_V/V}.\mathcal{O}_{\cX_V}=\mathcal{O}_{X_V}\text{ and, }\mathcal{O}_\cX(-\cE)|_{\cX_V}\simeq\mathcal{O}_{\cX_V}(-\cE_V)=\mathcal{O}_{\cX_V}.
 \end{align*}
 We therefore obtain $\left(f^{-1}\cI_{\cY/\cZ}.\mathcal{O}_\cX\right)|_{\cX_V}\simeq\mathcal{O}_\cX(-\cE)|_{\cX_V}$. Now we can restrict to opens in $U$, which intersect $\cY_U$. For any closed point $y\in \cY_U$, choose a connected open neighborhood $y\in U_y\subseteq U$ such that $\cY_{U_y}:=U_y\cap \cY_U$ is irreducible. One can do this by first restricting $U$ to the connected component $U_1$ containing $y\in Y_U$, and then removing the irreducible components of $Y_U$ which intersect $U_1$ but do not contain $y$, thereby obtaining a smaller open $y \in U_2\subseteq U_1$ and then taking connected component $U_y$ of $U_2$ containing $y$. This also implies the irreducibility of $\cE_{U_y}:=\cE\times_\cZ U_y$ as follows. Assume $\cE_{U_y}=\bigsqcup_{i=1}^{m}\cE_i$ is a decomposition into irreducible components. Since $f_{U_y}:\cX_{U_y}\rightarrow U_y$ is proper, has connected fibers and $f^{-1}(\cY_{U_y})=\cE_{U_y}$ set-theoretically, we obtain a decomposition $\cY_{U_y}=\bigsqcup_{i=1}^{m}f_{U_y}(\cE_i)$ into disjoint closed subsets, {since there can be at most one irreducible component of $\cE_{U_y}$ lying over a given $f_{U_y}(\cE_i)$ for any $1\leq i\leq m$}. By connectedness of $\cY_{U_i}$, it follows that $f_{U_y}(\cE_i)$ are empty for all but one $i$ and thus, so are $\cE_i$'s. This proves the irreducibility of $\cE_{U_y}$.

Thus, restricting $f:\cX\rightarrow \cZ$ to $U_y$, we obtain a projective birational contraction $f_y: \cX_{U_y}\rightarrow U_y$ such that the smooth irreducible exceptional Cartier divisor $\cE_{U_y}:=\cE\times_\cZ U_y$ is set-theoretically equal to the inverse image $f_y^{-1}(\cY_{U_y})$, where $\cY_{U_y}\subseteq U_y$ is a smooth closed irreducible subscheme and $\cX_{U_y}\setminus \cE_{U_y}\xrightarrow{\sim} U_y\setminus \cY_{U_y}$. By Theorem~\ref{scheme case}, it follows that $\cX_{U_y}\cong \Bl_{\cY_{U_y}}U_y$ is a blowup along the smooth center $\cY_{U_y}\hookrightarrow U_y$ and thus, $f_{U_y}^{-1}\cI_{\cY_{U_y}/U_y}.\mathcal{O}_{\cX_{U_y}}\simeq\mathcal{O}_{\cX_{U_y}}(-\cE_{U_y})$, as ideal sheaves on $\cX_{U_y}$ by \cite[\href{https://stacks.math.columbia.edu/tag/02OS}{Tag 02OS}]{stacks-project}. By compatibility with base change, this is equivalent to $f^{-1}\cI_{\cY/\cZ}.\mathcal{O}_\cX|_{\cX_{U_y}}\simeq\mathcal{O}_\cX(-\cE)|_{\cX_{U_y}}$, as ideal sheaves on $\cX_{U_y}$ for all $y\in \cY_U$. Combining these isomorphisms for all $y\in \cY_U$ and $V=U\setminus \cY_U$, we obtain $f^{-1}\cI_{\cY/\cZ}.\mathcal{O}_\cX|_{\cX_{U}}\simeq\mathcal{O}_\cX(-\cE)|_{\cX_{U}}$, as ideal sheaves on $\cX_{U}$. Since $\cX_U\rightarrow \cX$ is a smooth surjective e\'tale presentation, we obtain the desired isomorphism $f^{-1}\cI_{\cY/\cZ}.\mathcal{O}_\cX\simeq\mathcal{O}_\cX(-\cE)$, as ideal sheaves on $\cX$.

To show that $\mathcal{O}_\cX(-\cE)$ is relatively ample with respect to $f:\cX\rightarrow \cZ$, it suffices to do so at the level of \'etale local schematic presentations of the target $\cZ$ by Lemma~\ref{Lem:etalelocaltarget}. In particular, we can check $f$-ampleness of $\mathcal{O}_\cX(-\cE)$ by restricting over the etale local charts $U_y\rightarrow \cZ$ for all $y\in \cY_U$ and $V:=U\setminus \cY_U\rightarrow \cZ$ considered above, since these together produce a smooth surjective \'etale schematic chart of $\cZ$. Since $f_V:\cX_V\xrightarrow{\sim} V$ is an isomorphism, $\mathcal{O}_\cX(-\cE)|_{\cX_V}\simeq \mathcal{O}_{\cX_V}$ is $f_V$-ample. Finally for any $y\in \cY_U$, if $f_y: \cX_{U_y}\rightarrow U_y$ is the restriction of $f$ over $U_y$, then from the paragraph above we have $\cX_{U_y}\cong \Bl_{\cY_{U_y}}U_y$, whereby $\mathcal{O}_{\cX}(-\cE)|_{\cX_{U_y}}\simeq \mathcal{O}_{\cX_{U_y}}(1)$ is $f_y$-ample by \cite[\href{https://stacks.math.columbia.edu/tag/02OS}{Tag 02OS}]{stacks-project}. This completes the proof of $\mathcal{O}_\cX(-\cE)$ being $f$-relatively ample.
\end{proof}

\begin{Cor}\label{cor:pushforward}
In the situation of Proposition~\ref{ideal equality} we have the isomorphism $\cI^m_{\cY/\cZ}\simeq f_\ast\mathcal{O}_\cX(-m\cE)$ of ideal sheaves on $\cZ$ for all $m\geq 0$.
\end{Cor}

\begin{proof}
For $m=0$, the desired isomorphism is $f_\ast\mathcal{O}_\cX\simeq\mathcal{O}_\cZ$, which holds by contraction definition. By Proposition~\ref{ideal equality}, we have the isomorphism $f^{-1}\cI_{\cY/\cZ}.\mathcal{O}_\cX\simeq\mathcal{O}_\cX(-\cE)$, as ideal sheaves on $\cX$. Taking the $m$-th power for all $m\geq 2$ yields, $f^{-1}\cI^m_{\cY/\cZ}.\mathcal{O}_\cX\simeq(f^{-1}\cI_{\cY/\cZ}.\mathcal{O}_\cX)^m\simeq \mathcal{O}_\cX(-m\cE)$, since extension of ideals commutes with taking powers. By definition of $f^{-1}\cI^m_{\cY/\cZ}.\mathcal{O}_\cX$, we have a canonical surjective map $f^\ast\cI^m_{\cY/\cZ}\twoheadrightarrow f^{-1}\cI^m_{\cY/\cZ}.\mathcal{O}_\cX\simeq \mathcal{O}_\cX(-m\cE)$. By adjunction $f^\ast\dashv f_\ast$, we obtain a canonical map $\alpha_m: \cI^m_{\cY/\cZ}\rightarrow f_\ast\mathcal{O}_\cX(-m\cE)$ of ideal sheaves on $\cZ$. We claim that $\alpha_m$ is an isomorphism for all $m\geq 1$. Since isomorphism is an \'etale local property, we will check this on charts. In particular, we will take \'etale local charts as in proof of Proposition~\ref{ideal equality}: the fixed smooth surjective \'etale presentation $u:U\rightarrow \cZ$ and its refinements by opens $V=U\setminus \cY_U$ and open neighborhoods $y\in U_y\subseteq U$ for each closed point $y\in \cY_U$. Over $V:=U\setminus\cY_U$, the morphism $f_V:\cX_V\rightarrow V$ is an isomorphism since $\cY_V=\cE_V=\emptyset$ and hence by functoriality of $\alpha_m$,
$$\alpha_m|_V:\cI_{\cY/\cZ}^m|_V=\mathcal O_V\xrightarrow{\sim}(f_V)_*\mathcal O_{\cX_V}\simeq (f_\ast\mathcal{O}_\cX(-m\cE))|_V.$$ 
Thus $\alpha_m|_V$ is an isomorphism. For every $y\in\cY_U$, the proof of Proposition~\ref{ideal equality} yields the isomorphism $\cX_{U_y}\simeq \Bl_{\cY_{U_y}}U_y$ and $f_{U_y}:\cX_{U_y}\rightarrow U_y$ is the blowdown morphism. By flat base change,
\begin{equation}\label{Eq:restrictedsheaves}
(f_\ast\mathcal{O}_\cX(-m\cE))|_{U_y} \simeq
(f_{U_y})_*\mathcal O_{\cX_{U_y}}(-m\cE_{U_y}), \ \text{and} \ 
\cI_{\cY/\cZ}^m|_{U_y} \simeq \cI_{\cY_{U_y}/U_y}^{m}.
\end{equation}
The standard pushforward formula \eqref{hartog} for the blowup along a smooth center to $f_{U_y}:\cX_{U_y}\rightarrow U_y$ yields
\[
 \alpha_m|_{U_y}:\cI_{\cY/\cZ}^m|_{U_y} \simeq \cI_{\cY_{U_y}/U_y}^{m}\xrightarrow{\sim}
(f_{U_y})_*\mathcal O_{\cX_{U_y}}(-m\cE_{U_y})\simeq (f_\ast\mathcal{O}_\cX(-m\cE))|_{U_y}.
\]
Thus, $\alpha_m|_{U_y}$ is an isomorphism for every open $U_y\subseteq U$ for $y\in \cY_U$. Since $V$ together with the opens $U_y$ cover $U$, they form a surjective \'etale cover of $\cZ$, whereby it follows that $\alpha_m:\cI^m_{\cY/\cZ}\rightarrow f_\ast\mathcal{O}_\cX(-m\cE)$ is an isomorphism of ideal sheaves
on $\cZ$.
\end{proof}

\begin{Lemma}\label{lemma:stackyproj}
In the situation of Proposition~\ref{ideal equality}, we have that $\cX$ is isomorphic to $\Proj_\cZ\left(\bigoplus_{n\geq0}f_\ast\mathcal{O}_\cX(-n\cE)\right)$, which is the ordinary Proj of the graded algebra $\bigoplus_{n\geq0}f_\ast\mathcal{O}_\cX(-n\cE)$.
\end{Lemma}
\begin{proof}
     By Proposition~\ref{ideal equality} the line bundle $\cF:=\mathcal{O}_\cX(-\cE)$ on $\cX$ is relatively ample with respect to $f:\cX\rightarrow \cZ$, and the induced map $(f, \cF):\cX\rightarrow \cZ\times \mathcal{B}\GG_m$ is representable since $f$ is representable, and $\cF$ is uniformizing. Thus, by Proposition~\ref{Prop:amplephi} the morphism to the stacky proj $\varphi_\cF: \cX\rightarrow \sProj_\cZ\left(\bigoplus_{n\geq0}f_\ast\mathcal{O}_\cX(-n\cE)\right)$ as in \eqref{Eq:phi}, gives an isomorphism of stacks over $\cZ$. The relative coarse space of the stack $\sProj_\cZ\left(\bigoplus_{n\geq0}f_\ast\mathcal{O}_\cX(-n\cE)\right)$ over $\cZ$ is the ordinary proj construction $\Proj_\cZ\left(\bigoplus_{n\geq0}f_\ast\mathcal{O}_\cX(-n\cE)\right)$ by \cite[Prop.~1.6.1(iii)]{QR}. Since $f:\cX\rightarrow \cZ$ is representable,  by \cite[Prop.~1.6.2]{QR} it follows that $\cX\cong \Proj_\cZ\left(\bigoplus_{n\geq0}f_\ast\mathcal{O}_\cX(-n\cE)\right)$.
\end{proof}

We can now give the proof of the main theorem of this section.

\begin{proof}[Proof of Theorem~\ref{main representable}]
    Suppose $\cX$ and $\cZ$ are smooth separated Deligne--Mumford stacks of dimension $n$, and there is a representable contraction $f : \cX \to \cZ$ where the exceptional $\cE$ is a smooth irreducible Cartier divisor.  Moreover, suppose there is a smooth closed irreducible Deligne--Mumford substack $\cY \subseteq \cZ$ such that $(f^{-1}\cY) \simeq \cE$ set-theoretically, and $f : \cX \setminus \cE \xrightarrow{\sim} \cZ \setminus \cY.$ By Corollary~\ref{cor:pushforward} we have the isomorphism $\cI_{\cY/\cZ}^n\simeq f_\ast\mathcal{O}_\cX(-n\cE)$ of ideal sheaves on $\cZ$, for all $n\geq 0$. Thus, we have an isomorphism $\bigoplus_{n\geq 0}\cI^n_{\cY/\cZ}\simeq \bigoplus_{n\geq 0}f_\ast\mathcal{O}_\cX(-n\cE)$ of sheaves of graded $\mathcal{O}_\cZ$-algebras. By Lemma~\ref{lemma:stackyproj}, functoriality of the ordinary Proj construction, and definition of blowup along an ideal sheaf, we therefore obtain the following isomorphism of stacks over $\cZ$
    \[\cX\cong\Proj_\cZ\left(\bigoplus_{n\geq 0}f_\ast\mathcal{O}_\cX(-n\cE)\right)\cong \Proj_\cZ\left(\bigoplus_{n\geq 0}\cI^n_{\cY/\cZ}\right)\cong \Bl_\cY \cZ.\] 
    It is then a consequence of ordinary blowup that $f|_{\cE} : \cE \to \cY$ is isomorphic to the projective bundle $\PP(\cI_{\cY/\cZ}/\cI_{\cY/\cZ}^2)\rightarrow \cY$, and $\mathcal{O}_{\cX}(\cE)|_{\cE} \simeq \mathcal{O}_{\cE}(-1) \otimes f^*\cL$ for some line bundle $\cL$ on $\cY$.
\end{proof}

 \section{Regular Deligne--Mumford Surface Case}\label{section:DMsurface}

 In this section, we confirm Question~\ref{conje} when $\cX$ and $\cZ$ are smooth separated Deligne--Mumford surfaces.

 \begin{theorem}
 \label{main body}
    Suppose $\cX$ and $\cZ$ are smooth separated Deligne--Mumford surfaces, and there is a contraction $f : \cX \to \cZ$ where the exceptional $\cC$ is a smooth irreducible stacky curve. 
    Moreover, suppose $f(\cC) = p \in |\cZ|$ and $f : \cX \setminus \cC \xrightarrow{\sim} \cZ \setminus \cY$ is an isomorphism where $\cY := \mathcal{B}\Aut_{\cZ}(p)$. 
    Then, $f$ is a regular weighted blowup of $\cZ$ with reduced center $\cY$. In particular, $f|_\cC : \cC \to \cY$ is a weighted projective bundle, and $\mathcal{O}_{\cX}(\cC)|_{\cC} \simeq \mathcal{O}_{\cC}(-1) \otimes f^*\cL$ for some line bundle $\cL$ on $\cY$.
 \end{theorem}

By virtue of \cite[Prop.~2.18]{SWB2024}, a morphism of algebraic stacks being a weighted blowup is smooth local on the target, whereby  we may assume without loss of generality that $\cZ$ is a smooth irreducible variety.
Further, we can take $\cZ$ to be a smooth affine surface $\Spec \, A$.
To emphasize that the target is now a scheme we write $f :\cX \to Z.$

 \subsection{Contracted Cartier Divisors are simply-connected}
   \label{contracted cartier divisor}

  In this section, we prove that if $\cE := \Exc(\pi)$ is a smooth irreducible divisor where $\pi : \cX \to Z$ is a contraction from a smooth, separated Deligne--Mumford stack to a smooth irreducible variety where $\dim(\cX) = \dim(Z) > 1$ and $\pi(\cE) = p \in Z$, then $\cE$ is algebraically simply-connected.  
  To do so, we need the following results, which are known as topological invariance of the fundamental group.
  In what proceeds, we use the notation $\text{F\'Et}_X$  as defined in \cite[\href{https://stacks.math.columbia.edu/tag/0BL6}{Tag 0BL6}]{stacks-project}.

 \begin{Lemma}[{\cite[\href{https://stacks.math.columbia.edu/tag/0BQB}{Tag 0BQB}]{stacks-project}}]
  \label{topological inv}
   Let $X \subset X'$ be a thickening of schemes. The functor
   \[
   \text{F\'Et}_{X'} \to \text{F\'Et}_X, \ \ \ U' \mapsto U' \times_{X'} X
   \]
   is an equivalence of categories. 
 \end{Lemma}

 Topological invariance of the fundamental group easily extends to separated Deligne--Mumford stacks by descent. 
 Before we give the proof of Theorem~\ref{main} we will recall the {\em frame bundle} associated to a finite \'etale morphism of smooth Deligne--Mumford stacks.

\begin{Def}
 \label{frame bundle}
Let $\cY \to \cZ$ be a degree $d$ finite \'etale morphism of smooth separated Deligne--Mumford stacks. 
Then, the {\em frame bundle} $\underline{\mathrm{Isom}}_{\cZ}(\cZ^{\sqcup \, d}, \cY) \to \cZ$ defined as follows: suppose $T \to \cZ$ is an \'etale morphism where $T$ is a complex scheme, then
\[
\underline{\mathrm{Isom}}_{\cZ}(\cZ^{\sqcup \, d}, \cY)(T) := \underline{\mathrm{Isom}}_T(T^{\sqcup \, d}, \cY \times_{\cZ} T)
\]
where $\underline{\mathrm{Isom}}_T(T^{\sqcup \, d}, \cY \times_{\cZ} T)$ is the usual Isom scheme (see \cite[Exp. XIV, Sec. 1, pg. 164]{SGA2}).
The frame bundle naturally defines a principal $S_d$-bundle induced by the canonical left action of $S_d$ on $\cZ^{\sqcup \, d}.$
 \end{Def}

 Before stating the theorem we will recall {\em Tannaka duality}, which we need for the proof.

 \begin{theorem}[{\cite[Thm. 1.1]{MR4009673}}]
 \label{tannaka}
    Let $\cX$ be a noetheian algebraic stack with affine stabilizers.
    For every locally excellent algebraic stack $\cY$ the functor
    \[
    \Hom(\cY, \cX) \to \Hom_{r\otimes, \simeq}(\mathsf{Coh}(\cX), \mathsf{Coh}(\cY))  =: \Hom(\mathsf{Coh}(\cX), \mathsf{Coh}(\cY))
    \]
    is an equivalence, which we call Tannaka duality.
 \end{theorem}

 We will now state the key theorem for this section.

 \begin{theorem}
   \label{contracted}
    Suppose $\cE \hookrightarrow \cX$ is a smooth irreducible Cartier divisor, where $\cX$ is a smooth, separated Deligne--Mumford stack of dimension $n > 1$.
    Suppose also that there exists a contraction $\pi : \cX \to Z$ to a smooth irreducible variety of dimension $n$ such that $\pi(\cE) = p \in Z$ and $\pi : \cX \setminus \cE \xrightarrow{\sim} Z \setminus p$ is an isomorphism. 
    Then, $\cE$ is algebraically simply-connected, i.e., $\pi_1^{\text{\'et}}(\cE) \simeq 1$.  
 \end{theorem}   
 
 \begin{proof}
  Let $\gamma : \cE' \to \cE$ be a finite \'etale morphism. 
  Assume that $\gamma$ is non-trivial.
  We will proceed by way of contradiction to show $\pi_1^{\text{\'et}}(\cE) \simeq 1.$
  The proof consists of several steps.
  Note that by taking a chart around $p \in Z$ we may assume without loss of generality that $Z = \Spec \, A$ is a smooth affine scheme where $p \in Z$ corresponds to a maximal ideal $\mathfrak{m} \subseteq A.$ 
   \step 
   First, we extend $\gamma$ to a finite \'etale morphism along a sequence of thickenings of the stack $\mathcal{E}.$
   More specifically,  set $\cE_{n} := \cX \times_{\Spec \, A} \Spec \, A/\mathfrak{m}^{n+1}$ for each $n \ge 0$. 
   Then, the natural inclusion $\cE_n \hookrightarrow \cE_{n+1}$ is a thickening of stacks for every $n \ge 0$ whose underlying reduced stack is $\mathcal{E}$.
   Hence, by topological invariance of the fundamental group we can uniquely extend $\gamma : \cE' \to \cE$ to a finite \'etale morphism $\gamma_0 : \cE'_0 \to \cE_0$ where $\cE' \simeq \cE \times_{\cE_0} \cE'_0$.
   By successive application of the topological invariance of the fundamental group we obtain a finite \'etale morphism $\gamma_n : \cE'_n \to \cE_n$ for each $n \ge 0$ such that $\cE'_n \simeq \cE_n \times_{\cE_{n+1}} \cE'_{n+1}$
  \[
  \begin{tikzcd}
  \cE' \arrow[d, "\gamma"] \arrow[r, hook] & \cE'_0 \arrow[d, "\gamma_0"] \arrow[r, hook] & \cE'_{1} \arrow[d, "\gamma_{1}"] \arrow[r, hook] & \cE'_{2} \arrow[d, "\gamma_{2}"] \arrow[r, hook] & \cdots \arrow[r, hook] & \cE'_{n} \arrow[d, "\gamma_{n}"] \arrow[r, hook] & \cE'_{n+1} \arrow[d, "\gamma_{n+1}"] \arrow[r, hook] & \cdots \\
  \cE \arrow[r, hook]  & \cE_0 \arrow[r, hook]                             & \cE_{1} \arrow[r, hook]                                   & \cE_{2} \arrow[r, hook]                                   & \cdots \arrow[r, hook] & \cE_{n} \arrow[r, hook]                                   & \cE_{n+1} \arrow[r, hook]                                     & \cdots
  \end{tikzcd}
 \]
   Thus, we have our desired sequence of finite \'etale morphisms $\{\gamma_n\}_{n \ge 0}$ along thickenings of the stack $\cE$. \step
   Next, we use the sequence of finite \'etale morphisms $\{\gamma_n\}_{n \ge 0}$ to obtain a sequence of morphisms $\{\cE_n \to BS_d\}_{n \ge 0}.$
   To each $\gamma_n : \cE'_n \to \cE_n$ we associate the  {frame bundle} (see Definition~\ref{frame bundle})
   \[
   P_n := \underline{\operatorname{Isom}}_{\cE_n}(\cE_n^{\sqcup \,d}, \cE'_n) \to \cE_n.
   \] 
   which is a principal $S_d$-bundle.
   Hence, the frame bundle gives a morphism $\cE_n \to BS_d$, giving the desired sequence. \step
   Next, we extend the sequence $\{\cE_n \to BS_d\}_{n \ge 0}$ to a principal $S_d$-bundle over 
   \[
   \hat{\cX} := \cX \times_{\Spec \, A} \, \Spec \, \hat{A}
   \]
   where $\hat{A}$ is the completion with respect to $\mathfrak{m}$, i.e., $\hat{A} = \lim\limits_n A/\mathfrak{m}^n.$
   To extend $\{\cE_n \to BS_d\}_{n \ge 0}$ to $\hat{\cX} \to BS_d$ we will use {Tannaka duality}. Indeed, since a formally smooth Deligne--Mumford stack is locally excellent, and a Deligne--Mumford stack has affine stabilizers we have
   \[
    \Hom(\hat{\cX}, BS_d) \simeq \Hom(\mathsf{Coh}(BS_d), \mathsf{Coh}(\hat{\cX})).
   \]
   However, by \cite[Thm. 1.4]{MR2183251} the pair $(\cE_0, \hat{\cX})$ is coherently complete, i.e., $\lim\limits_n \mathsf{Coh}(\cE_n) \simeq \mathsf{Coh}(\hat{\cX}).$
   Hence, since $\Hom(A, -)$ preserves inverse limits
   \[
    \Hom(\mathsf{Coh}(BS_d), \mathsf{Coh}(\hat{\cX}))
    \simeq
    \Hom(\mathsf{Coh}(BS_d), \lim_n\mathsf{Coh}({\cE_n}))
    \simeq
    \lim_n \Hom(\mathsf{Coh}(BS_d), \mathsf{Coh}({\cE_n})).
   \]
   By another application of Tannaka duality
   \[
    \Hom(\mathsf{Coh}(BS_d), \mathsf{Coh}({\cE_n})) \simeq \Hom(\cE_n, BS_d)
   \]
   and by combining the isomorphisms, we see that
   \begin{equation}
   \label{td}
    \Hom(\hat{\cX}, BS_d)
    \simeq
    \lim_n \Hom(\cE_n, BS_d)
   \end{equation}
   Therefore, by (\ref{td}) the system $\{\cE_n \to BS_d\}_{n \ge 0}$ extends to define $\hat{\cX} \to BS_d$, which is equivalent to a principal $S_d$-bundle 
   $P \to \hat{X}$. 
   Then, by considering the associated $d$-sheeted cover $\hat{\cX}' := P \times^{S_d} \{1, \ldots, d\}$ we obtain a degree $d$ finite \'etale morphism
   $g: \hat{\cX}' \to \hat{\cX},$
   where by construction $g |_{\cE'} = \gamma.$ 
   \step Now, we will show that the induced restriction morphism $g : \hat{\cX}' \setminus \cE' \to \hat{\cX} \setminus \cE$ is a trivial cover, and hence,
   \[
    \hat{\cX}' \setminus \cE' \simeq (\hat{\cX} \setminus \cE)^{\sqcup \, d}.
   \]
   Indeed, the restriction morphism $g : \hat{\cX}' \setminus \cE' \to \hat{\cX} \setminus \cE$ is a finite \'etale morphism, and by basechange $\hat{\cX} \setminus \cE \simeq (\Spec \,\hat{A}) \setminus p$.
   Thus, the composition morphism $\pi \circ g : \hat{\cX}' \setminus \cE' \to (\Spec \, \hat{A}) \setminus p$ is a finite \'etale morphism.
   Now, the complete local ring $\hat{A}$ is regular of $\dim \Spec \, \hat{A} \ge 2$, and the point $p \in \, \Spec \,\hat{A}$ corresponds to the maximal ideal, so it follows from the Zariski-Nagata purity theorem \cite[Exp. X, Thm. 3.4]{SGA2} that there is an equivalence of categories 
   \[
    \text{F\'Et}((\Spec \, \hat{A}) \setminus p) \simeq \text{F\'Et}(\Spec \, \hat{A}).
   \]
   However, since $\hat{A}$ is strictly henselian with residue field $\CC$,
   \[
    \text{F{\'E}t}(\Spec \, \hat{A}) \simeq \text{F\'Et}(\Spec \, \CC).
   \]
   Since $\CC$ is algebraically closed, $\pi_1^{\text{\'et}}(\Spec \, \CC) =1$, and thus, every finite \'etale morphism is a trivial cover.
   In other words, $\pi \circ g : \hat{\cX}' \setminus \cE' \to (\Spec \, \hat{A}) \setminus p$ is a trivial cover, which implies $g : \hat{\cX}' \setminus \cE' \to \hat{\cX} \setminus \cE$ is a trivial cover, i.e.,
   \[
    \hat{\cX}' \setminus \cE' \simeq (\hat{\cX} \setminus \cE)^{\sqcup \, d}.
   \]
   as desired.
   
   \step Lastly, we will show $g : \hat{\cX}' \to \hat{\cX}$ is trivial, which implies $\gamma : \cE' \to \cE$ is trivial, achieving the desired contradiction.
   
   To prove $g : \hat{\cX}' \to \hat{\cX}$ is a trivial cover we use the notion of relative normalization defined in Definition \ref{rel norm}. First, let $U := \hat{\cX}' \setminus \mathcal{E}'$ denote the open dense subscheme in $\hat{\cX}'.$ 
   The morphism $g|_U : U \to \hat{\cX}$ induced by $g$ is \'etale and has image $g(U) \simeq \hat{\cX} \setminus \cE,$ which is an open dense subscheme of $\hat{\cX}.$ 
   By Proposition \ref{rel norm in open subscheme} the normalization of $\hat{\cX}$ in $g|_U : U \to \hat{\cX}$ is $g : \hat{\cX}' \to \hat{\cX},$ i.e., $\hat{\cX}_{g|_U}^\nu \simeq \hat{\cX}'.$
   
   On the other hand, by Remark \ref{reduced not int} the normalization of $\hat{\cX}$ in $g|_U : U \to \hat{\cX}$ is precisely the disjoint union of $d$-copies of the normalization of $\hat{\cX}$ in $\hat{\cX}\setminus \cE.$
   Since $\hat{\cX} \setminus \cE \subseteq \hat{\cX}$ is an open dense subscheme and $\hat{\cX}$ is normal it follows from Proposition \ref{norm in open subscheme} that the normalization of $\hat{\cX}$ in $\hat{\cX}\setminus \cE$ is $\hat{\cX}.$ 
   Hence, the normalization of $\hat{\cX}$ in $g|_U:U \to \hat{\cX}$ is the trivial cover $ \chi : \hat{\cX}^{\,\sqcup \, d} \to \hat{\cX},$ i.e., $\hat{\cX}^\nu_{g|_U} \simeq \hat{\cX}^{\,\sqcup \, d}.$
   
   Since relative normalization is unique up to unique isomorphism it follows that  $\hat{\cX}' \simeq \hat{\cX}^{\,\sqcup \, d}$ and $g = \gamma$ is a trivial cover.
 \end{proof}

 Theorem \ref{contracted} is the key result needed to prove Theorem \ref{main body}. Indeed, we immediately obtain the following corollary.

\begin{Cor}
Suppose $f : \cX \to \cZ$ is the contraction morphism from Theorem \ref{main body}, which contracts a smooth stacky curve $\cC.$ 
Then, $\cC$ is algebraically simply-connected, i.e., $\pi^{\text{\'et}}_1(\cC) \simeq 1.$
\end{Cor}  

 We utilize this in Section \ref{sc stacky curves} to prove $\cC \simeq \cP(a,b)$ by classifying all algebraically simply-connected smooth stacky curves.

 \subsection{Simply-connected Stacky Curves}
 \label{sc stacky curves}

 In this section, we prove that the only algebraically simply-connected smooth stacky curves are weighted projective curves $\cP(a,b)$ where $a,b > 0$.
 Note, we are only considering smooth stacky curves whose coarse moduli space is projective.
 
 \begin{theorem} \label{esc stacky curves}
    Suppose $\cC$ is an algebraically simply-connected smooth stacky curve.
    Then, $\cC \simeq \mathcal{P}(a, b)$ is a weighted projective curve for some $a,b > 0.$
 \end{theorem} 

 Since contracted smooth divisors are algebraically simply connected by Theorem \ref{contracted} we immediately obtain the following corollary.

 \begin{Cor}
   \label{contracted curve is weighted projective}
   Let $\cC$ be a contracted smooth stacky curve as in the statement of Theorem \ref{main body}. 
   Then, $\cC \simeq \cP(a,b)$ for some $a,b > 0.$
 \end{Cor}

 To prove Theorem \ref{esc stacky curves} there are two cases to consider, namely when $\cC$ has trivial generic stabilizer, and when $\cC$ has a non-trivial generic stabilizer. 
 It turns out we can reduce to the case $\cC$ has a trivial generic stabilizer. 
 To see this, we need to compare the algebraic fundamental group to the topological fundamental group.
 We first note that the coarse moduli space $C$ of $\cC$ is necessarily $\PP^1$.
 
 \begin{Lemma}
 \label{P1 is cms}
  Suppose $\cC$ is an algebraically simply-connected smooth stacky curve. 
  Then, $C \simeq \PP^1$ where $C$ is the coarse moduli space of $\cC.$
 \end{Lemma}

 \begin{proof}
  Suppose $X \to C$ is a non-trivial finite \'etale morphism of degree $d > 1.$ 
  Consider the fiber product $\cX := \cC \times_C X.$ Since $\pi_1^{\text{\'et}}(\cC) \simeq 1$ the induced finite \'etale morphism $\cX \to \cC$ is trivial.
  That is, $\cX \simeq \cC^{\sqcup \,d}$ and the morphism to $\cC$ is the canonical projection.
  Observe that the projection morphism $\pi_{\cX} : \cX \to X$ is a coarse moduli space by virtue of basechanging over a flat morphism \cite[Thm. 11.1.2 (iii)]{MR3495343}.
  Since $\pi_{\cC} : \cC \to C$ is the coarse moduli space and $\cX \simeq \cC^{\sqcup \, d}$ it is clear that $X \simeq C^{\sqcup \, d}$ and $X \to C$ is the standard projection.
  In other words $X \to C$ is trivial, which is a contradiction. Therefore, $\pi_1^{\text{\'et}}(C) = 1$, and since $C$ is defined over $\CC$ it is well-known that $\pi_1^{\text{\'et}}(C)^{\mathrm{ab}} \simeq (\hat{\ZZ})^{2g}$ where $g$ is the genus of $C$. Thus, $g=0$ and $C \simeq \PP^1.$
 \end{proof}

Behrend and Noohi proved the analogous result of Theorem \ref{esc stacky curves} for smooth analytic Deligne--Mumford curves, which we state here.

\begin{theorem}[{\cite[Thm. 1.1]{MR2279100}}]
The simply-connected smooth analytic proper Deligne--Mumford curves are precisely $\mathcal{P}(a,b)$ where $a,b \ge 1.$
\end{theorem}
 
Recall that for a scheme $X$ of finite type over $\CC$ there is a naturally associated complex analytic space $X^{\mathrm{an}}$ by GAGA. 
It is well known that the algebraic fundamental group for $X$ is the profinite completion of the topological fundamental group for $X^{\mathrm{an}},$ i.e., $\pi_1^{\text{\'et}}(X) \cong \hat{\pi_1(X^{\mathrm{an}})}$. 
Fortunately, there is an analogous result for algebraic stacks.

\begin{theorem}[{\cite[Cor. 20.5]{Noohi2005}}] \label{profinite}
Let $\cX$ be a connected algebraic stack of finite type over $\CC$. 
Then the algebraic fundamental group of $\cX$ is isomorphic to the profinite completion of the topological fundamental group of the underlying analytic stack $\cX^{\mathrm{an}}$.
\end{theorem}

\begin{Notation}
We use $\pi_1^{\text{\'et}}(\cX)$ to denote the algebraic fundamental group, and $\pi_1(\cX) := \pi_1(\cX^{\mathrm{an}})$ to denote the topological fundamental group of the underlying topological stack so that
$\pi_1^{\text{\'et}}(\cX) \cong \hat{\pi_1(\cX)}.$
\end{Notation}

With the relationship between the topological and algebraic fundamental groups we can show that a given weighted projective curve is algebraically simply-connected.

\begin{Prop}
  \label{weighted projective curve}
  Every weighted projective curve $\cP(a,b)$ is algebraically simply-connected. 
  That is, $\pi_1^{\text{\'et}}(\cP(a,b)) \simeq 1.$ 
\end{Prop}  

\begin{proof}
  By Theorem \ref{profinite} it is enough to show $\pi_1(\cP(a,b)) \simeq 1$ where we view $\cP(a,b) \simeq [(\CC^2 \setminus 0)/\CC^*]$ as an analytic stack.
  The canonical projection $\CC^2 \setminus 0 \to \cP(a,b)$ defines a $\CC^*$-fibration and by considering the associated homotopy long exact sequence it follows that $\pi_1(\cP(a,b)) \simeq 1.$
\end{proof}  
 
\subsection*{Non-trivial Generic Stabilizer}

Let $\cC$ be an algebraically simply-connected smooth stacky curve with non-trivial generic stabilizer $H$.
By \cite[Thm. A.1 and Ex. A.3]{MR2427954} the {\em rigidification} $\hat{\cC}  := \cC \sslash H$ defines a smooth stacky curve with trivial generic stabilizer. 
Further, the projection morphism $\cC \to \hat{\cC}$ defines a {\em $H$-gerbe}, which we now recall the definition relevant for our purposes.

 \begin{Def}[$G$-gerbe]
Let $G$ be a finite group scheme, and suppose $\cY \to \cZ$ is a morphism between two smooth separated Deligne--Mumford stacks. 
Then, equipping $\cZ$ with the \'etale topology, the morphism $\cY \to \cZ$ is a {\em $G$-gerbe} if there exists an \'etale covering $\{Z_i \to \cZ\}$ such that there is an isomorphism $\cY \times_{\cZ} Z_i \simeq BG$ over $\cZ/Z_i$ for every $i$.
\end{Def}

Now, we will show that $\pi_1^{\text{\'et}}(\cC) \simeq 1$ forces $\pi_1^{\text{\'et}}(\hat{\cC}) \simeq 1$.

\begin{Prop}
  Let $\cC$ be a smooth stacky curve as above. 
  Consider the rigidification $\cC \to \hat{\cC}$ by the generic stabilizer of $\cC$.
  Then, $\pi_1^{\text{\'et}}(\cC) \simeq 1$ forces $\pi_1^{\text{\'et}}(\hat{\cC}) \simeq 1.$
\end{Prop}  

\begin{proof}
Let $f : \cY \to \hat{\cC}$ be a finite \'etale morphism and consider the fiber product $\cX := \cC \times_{\hat{\cC}} \cY$.
Then, the projection morphism $\cX \to \cC$ is a finite \'etale morphism, and hence $\cX \simeq \cC^{\sqcup \, d}.$ 
Moreover, since $\cC \to \hat{\cC}$ is a gerbe we see that $\cX \to \cY$ is a gerbe. 
Since gerbes have isomorphic coarse moduli spaces it follows that $\cY$ has coarse moduli space $Y \simeq \hat{C}^{\sqcup \, d}$ and the induced morphism $\bar{f} : Y \to C$ is a trivial finite \'etale morphism.
Hence, we may write $\cY \simeq \cY_1 \sqcup \cdots \sqcup \cY_d$ where each $\cY_i$ is a smooth connected stacky curve with coarse moduli space $Y_i \simeq \hat{C}.$
Thus, the induced morphism $\cY_i \to \hat{C}$ is finite \'etale for each $i$ where the induced coarse moduli space morphism $Y_i \xrightarrow{\simeq} \hat{C}$ is an isomorphism.
Since $\hat{C}$ contains a dense open subscheme it follows that $\cY_i \to \hat{C}$ is generically of degree $1$. However, since degree is locally constant and $\hat{\cC}$ is connected, $\cY_i \to \hat{\cC}$ is degree $1$, i.e., an isomorphism. Therefore, $\cY_i \simeq \hat{\cC}$ for each $i$, and hence, $\cY \simeq \hat{\cC}^{\sqcup \, d}$ where $\cY \to \cC$ is a trivial cover. 
In other words, $\hat{\cC}$ is algebraically simply-connected, i.e., $\pi_1^{\text{\'et}}(\hat{\cC}) = 1.$ 
\end{proof}  

Therefore, $\hat{\cC}$ is an algebraically simply-connected smooth stacky curve with trivial generic stabilizer. 
Assuming $\hat{\cC} \simeq \cP(a,b)$ for some $a, b > 0$ with $\gcd(a,b) = 1$ we can show that $\cC$ is a weighted projective curve. 
We will use the following bijection.

\begin{Prop}[{\cite[Prop. 6.2]{MR2279100}}]
  \label{bijection}
Let $m, n$ be relatively prime positive integers. 
Let $H$ be a finite group and denote the center by $A := Z(H)$. 
Then, there is a natural bijection
\[
\{\text{Algebraic stacks $\cC$ that are $H$-gerbes over $\cP(a,b)$}\}_{/_{\mathrm{iso}}} \simeq {_{\operatorname{Out}(H)}\backslash \, A}.
\]
The fundamental group of the stack $\cC_a$ corresponding to an element $a \in H$ is (non-canonically) isomorphic to $H/a$, i.e., $\pi_1(\cC_a) \cong H/a$.
\end{Prop}

To see that $\cC$ defines a weighted projective curve, recall that $\cC \to \cP(a,b)$ defines a $H$-gerbe where $\pi_1^{\text{\'et}}(\cC) \simeq 1$. 
By Proposition \ref{bijection}, $\pi_1(\cC) \simeq H/a$ for some $a \in H.$
Since the profinite completion of a finite group is isomorphic to itself it follows that $H \simeq \mu_d$ is cyclic with $a$ a generator.
Further, since $\Aut(H)$ acts transitively on generators on $H$ it follows that $\cC$ is the unique algebraically simply-connected $\mu_d$-gerbe over $\cP(a,b).$
Now, notice that $\mathcal{P}(ad, bd)$ defines a $\mu_d$-gerbe over $\cP(a,b)$, and it follows from Proposition \ref{weighted projective curve} that $\cC \simeq \mathcal{P}(ad, bd)$, as desired.
Therefore, it is enough to show that the only algebraically simply-connected smooth stacky curves with generic trivial stabilizer are $\cP(a,b)$ with $\gcd(a,b) = 1.$

\subsection*{Trivial Generic Stabilizer}

Let $\cC$ be a smooth stacky curve with trivial generic stabilizer and let $C$ denote the coarse moduli space. 
To compute the algebraically simply-connected $\cC$ we will utilise Geraschenko and Satriano's bottom up characterization \cite[Thm. 1]{MR3719470}.
The bottom up characterization explicitly states how to recover a smooth Deligne--Mumford stack with trivial generic stabilizer from its coarse moduli space using a sequence of root stacks and canonical coverings. 
For stacky curves the algorithm, which we state below, is simpler than that for higher dimensions since we do not need canonical coverings. 

\begin{Prop}
   \label{Bottom up}
Let $\pi : \cC \to C$ denote the coarse moduli space, and let $D = \sum_i a_i p_i$ be the ramification divisor in $C$.
That is, $p_i \in C$ denotes the images of the stacky points and $a_i = |G_{p_i}|.$ Then, we can recover $\cC$ by a sequence of root stacks along the coarse moduli space
\[
\cC \simeq C(\sqrt{D}) := C(\sqrt[\leftroot{-2}\uproot{2}a_1]p_1)(\sqrt[\leftroot{-2}\uproot{2}a_2]p_2) \cdots (\sqrt[\leftroot{-2}\uproot{2}a_{m}]p_m).
\]
\end{Prop}

\begin{example}
   Consider the weighted projective curve $\cP(a,b)$ where $a,b > 1$ and $\gcd(a,b) = 1.$
   The coarse moduli space is $\PP(a,b) \simeq \PP^1$ and there are two stacky points at $0$ and $\infty$ with stabilizer groups $\mu_b$ and $\mu_a$ respectively.
   By Proposition \ref{Bottom up}
   \[
   \cP(a,b) \simeq \PP^1(\sqrt[\leftroot{-2}\uproot{2}b]0)(\sqrt[\leftroot{-2}\uproot{2}a]\infty).
   \]
\end{example}   

\begin{example}
   \label{P1/d}
Consider the $\mu_d$ action on $\PP^1$ defined by $(x,y) \mapsto (\zeta_d x, y)$ where $\zeta_d := \exp(2\pi i /d).$ 
Then, the quotient stack $[\PP^1/\mu_d]$ has coarse moduli space $\PP^1$ with two stacky points at $0$ and $\infty$ each with stabilizer group $\mu_d$.
By Proposition \ref{Bottom up}
\[
[\PP^1/\mu_d] \simeq \PP^1(\sqrt[\leftroot{-2}\uproot{2}d]0)(\sqrt[\leftroot{-2}\uproot{2}d]\infty).
\]
Observe also, that the projection $\PP^1 \to [\PP^1 / \mu_d]$ is a principal $\mu_d$-bundle. Hence, the projection defines a finite \'etale morphism, which is not an isomorphism.
Thus, $[\PP^1/\mu_d]$ is an example of a non-algebraically simply-connected smooth stacky curve. 
\end{example}   

We can now prove Theorem \ref{esc stacky curves} using Proposition \ref{Bottom up} and by computing the topological fundamental group and applying Theorem \ref{profinite}.

\begin{proof}[Proof of Theorem \ref{esc stacky curves}]
Let $\cC$ be an algebraically simply-connected smooth stacky curve where we may assume $\cC$ has trivial generic stabilizer. 
By Lemma \ref{P1 is cms} the coarse moduli space is $\PP^1.$
Note, if $\cC$ has no stacky points then $\cC \simeq \PP^1$.
Hence, we may as well assume $\cC$ has at least one stacky point.
Then, by Proposition \ref{Bottom up} we can write $\cC$ as a composition of root stacks
\[
\cC \simeq \PP^1(\sqrt[\leftroot{-2}\uproot{2}a_1]p_1)
(\sqrt[\leftroot{-2}\uproot{2}a_2]p_2) \cdots (\sqrt[\leftroot{-2}\uproot{2}a_{m}]p_m).
\]
Now, we will compute the topological fundamental group of the underlying orbifold $\cC^{\mathrm{an}}$. 
By the orbifold analogue of Seifert-Van Kampen theorem \cite[Thm. 4.7.1]{MR2962023} one can compute
\begin{equation} \label{tfg}
\pi_1(\cC^{\mathrm{an}}) = \left \langle \gamma_1, \ldots \gamma_m \, | \, \gamma_1^{a _1}= 1, \ldots, \gamma_m^{a_m} = 1, \prod_{i=1}^m \gamma_i = 1 \right \rangle.
\end{equation}
Using Theorem \ref{profinite} we can compute the \'etale fundamental group by taking the profinite completion. 
We proceed in three cases according to the number of stacky points.
\case {\em $1$ stacky point.} Suppose $\cC$ has $1$ stacky point. By Proposition \ref{Bottom up} it follows that $\cC \simeq \mathcal{P}(1, b)$ where $b > 1,$ which is algebraically simply-connected by Proposition \ref{weighted projective curve}.

\case {\em $2$ stacky points.} Suppose $\mathcal{C}$ has $2$ stacky points
with stabilizer groups $\mu_a$ and $\mu_b$, i.e., $\cC \simeq \PP^1(\sqrt[\leftroot{-2}\uproot{2}a]0)(\sqrt[\leftroot{-2}\uproot{2}b]\infty)$. We will consider two cases: $\gcd(a,b) = 1$, and $\gcd(a,b) = d > 1.$ For the first case, we see $\cC \simeq \cP(a,b)$ by Proposition \ref{Bottom up}, which again is algebraically simply-connected by Proposition \ref{weighted projective curve}.
If instead $\gcd(a,b) = d > 1$, then by (\ref{tfg})
\[
\pi_1(\cC^{\textrm{an}}) \simeq \langle \alpha, \beta \, | \, \alpha^{a} = 1, \beta^b = 1, \alpha\beta = 1 \rangle.
\]
Since $\gcd(a,b) = d$ it is straightforward to see that $\pi_1(\cC^{\textrm{an}}) \simeq \mu_d$, and since the profinite completion of a finite group is the group itself, $\pi_1^{\text{\'et}}(\cC) \simeq \mu_d$. Therefore, $\cC$ is not algebraically simply-connected, a contradiction. 

\case {\em $3$ or more stacky points.} Suppose $\cC$ has three or more stacky points. By (\ref{tfg})
\[
\pi_1(\cC^{\mathrm{an}}) = \left \langle \gamma_1, \ldots \gamma_m \, | \, \gamma_1^{a _1}= 1, \ldots, \gamma_m^{a_m} = 1, \prod_{i=1}^m \gamma_i = 1 \right \rangle
\]
where $m \ge 3.$

This group contains a proper normal finite index subgroup, see \cite{MR53937}.
Hence, the profinite completion is non-trivial, i.e., $\pi_1^{\text{\'et}}(\cC)$ is non-trivial.
Therefore, $\cC$ is not algebraically simply-connected, a contradiction.
\end{proof}

\subsection{Negativity of Contracted Smooth Stacky Curve}
\label{neg}
 
In this section, we will show $\mathcal{O}_\cX(\mathcal{C})|_{\cC} \simeq \mathcal{O}_\cC(-1)$ 
by first proving $\mathcal{C}$ is negative, i.e., $\mathcal{O}_\cX(\mathcal{C})|_{\cC} \simeq \mathcal{O}_\cC(-m)$ for some $m>0$, and then arguing $m=1.$ 
To prove negativity we will utilize the coarse moduli spaces $X$ and $C$ of $\cX$ and $\cC$ respectively and apply classical intersection theory. 
We again use $f : \cX \to Z$ to denote the contraction where we recall that $Z \simeq \Spec \, A$ is a smooth affine surface.   

\begin{Prop}
\label{-m}
The contracted smooth stacky curve $\cC$ is negative, i.e., $\mathcal{O}_{\cX}(\cC)|_{\cC} \simeq \mathcal{O}_{\cC}(-m)$ for some $m > 0.$
\end{Prop}   
\begin{proof}
Recall that $\cC \simeq \cP(a,b)$ for some $a,b \ge 1$ by Corollary \ref{contracted curve is weighted projective}. Hence, the statement $\Pic(\cC) \simeq \ZZ$ and $\mathcal{O}_{\cX}(\cC)|_{\cC} \simeq \mathcal{O}_{\cC}(k)$ for some $k \in \ZZ.$
Now, $\mathrm{sgn}(k) = \mathrm{sgn}(\cC \cdot \cC)$ where the intersection number is precisely $\cC \cdot \cC = \deg_{\cX}([\cC] \cdot [\cC])$ where we view $[\cC] \in A_*(\cX)_{\QQ}$ in the Chow ring. To compute the intersection product we may descend to the coarse space. Indeed, by \cite[Prop. 6.1]{MR1005008} the pushforward morphism $\pi_* : A_*(\cX)_\QQ \to A_*(X)_{\QQ}$ defines an isomorphism of rings where $\pi : \cX \to X$ denotes the coarse moduli space morphism. We also denote the coarse moduli space of $\cC$ by $C$. Let $d := \gcd(a,b)$ so that $\deg(\cC/C) = 1/d$, and hence $\pi_*[\cC] = \frac{1}{d} C$ (see \cite[Def. 3.6. (iii)]{MR1005008}). The pullback $\pi^* = (\pi_*)^{-1}$ is the inverse morphism, and hence, $\pi^*[C] = d[\cC]$. Then, by the projection formula it follows that
\[
\pi_*([\cC] \cdot [\cC]) = \pi_* \left ( [\cC] \cdot \frac{1}{d}\pi^*[C] \right ) = \frac{1}{d^2} \, [C]\cdot[C].
\]
By taking degrees we see that $\cC \cdot \cC = (C^2)_X /d^2.$
Thus it is enough to show that $C^2 < 0.$
To do so, let $p : \tilde{X} \to X$ be a resolution of singularities. By the projection formula
\[
C^2 = (p_*p^*C) \cdot C = (p^*C)^2.
\]
Moreover, $p^*C$ is effective and exceptional for the projective birational morphism $\bar{f} \circ p : \tilde{X} \to Z$ where $\bar{f} : X \to Z$ is the induced morphism on the coarse spaces. Therefore, by the relative Hodge index theorem, $C^2 = (p^*C)^2 < 0$.
In other words, $\cC \cdot \cC < 0$, and $\mathcal{O}_{\cX}(\cC)|_{\cC} \simeq \mathcal{O}_{\cC}(-m)$ for some $m > 0.$
\end{proof}   

To finish the proof we require the following result of Stephen Obinna.

\begin{theorem}[{\cite[Thm. A.1]{SWB2024}}]
   \label{-m case}
Suppose $\cX$ is a smooth and separated Deligne--Mumford stack, and $\cE \hookrightarrow \cX$ is a Cartier divisor such that $\cE \simeq \mathcal{P}(a_1, \ldots, a_r).$ 
Let $\mathcal{N} = \mathcal{O}_{\cX}(\cE)|_{\cE}$, and set $\cX_m := \cX(\sqrt[\leftroot{-2}\uproot{2}m]\cE)$ and $\cE_m := \cC(\sqrt[\leftroot{-2}\uproot{2}m]{\mathcal{N}}).$
Then, there exists a smooth and separated Deligne--Mumford stack $\cZ_m$ completing the diagram below, and such that $\cX_m$ can be recovered as a weighted blowup of $\cZ_m$ at $B\mu_m$ with exceptional divisor $\cE_m$.
\[
\begin{tikzcd}
\cE_m \arrow[rr, hook] \arrow[rd] \arrow[dd] &                                 & \cX_m \arrow[rd] \arrow[dd, dotted] &     \\
                                             & \cE \arrow[rr, hook] \arrow[dd] &                                     & \cX \\
B\mu_m \arrow[rd] \arrow[rr, dotted, hook]   &                                 & \cZ_m                               &     \\
                                             & p                               &                                     &    
\end{tikzcd}
\]
\end{theorem}  

\begin{Prop}
   \label{-1 curve}
The contracted stacky curve $\cC$ is a $(-1)$-curve, i.e., $\mathcal{O}_{\cX}(\cC)|_{\cC} \simeq \mathcal{O}_{\cC}(-1).$ 
\end{Prop} 

\begin{proof}
   By Proposition \ref{-m} we have $\mathcal{O}_{\cX}(\cC)|_{\cC} \simeq \mathcal{O}_{\cC}(-m)$ where $m > 0$. By Theorem \ref{-m case} taking the $m$-th root stack $\cX_m$ of $\cX$ along $\cC$ there exists a smooth separated Deligne--Mumford stack $\cZ_m$ such that $\cX_m$ can be recovered as the weighted blowup of $\cZ_m$ at a point $B\mu_m.$
   Moreover, by Theorem \ref{existence theorem}, there exists a unique morphism $g : \cZ_m \to Z$ completing the following diagram.
   \[
\begin{tikzcd}
\cC_m \arrow[rr, hook] \arrow[rd] \arrow[dd] &                                 & \cX_m \arrow[dd] \arrow[rd]                             &                     \\
                                             & \cC \arrow[rr, hook] \arrow[dd] &                                                         & \cX \arrow[dd, "f"] \\
B\mu_m \arrow[rr, hook] \arrow[rd]           &                                 & \cZ_m \arrow[rd, "{\exists ! \, g}"] \arrow[dd] &                     \\
                                             & p \arrow[rr, hook]              &                                                         & Z                   \\
                                             &                                 & Z_m \arrow[ru, "\bar{g}"]                               &                    
\end{tikzcd}
   \]
Note that the morphism $g$ is bijective and induces $\cZ_m \setminus B\mu_m \simeq Z \setminus p$. Passing to the coarse moduli space $Z_m$ of $\cZ_m$ defines a morphism $\bar{g} : Z_m \to Z.$ 
By a diagram chase observe that $\bar{g}$ is bijective and defines an isomorphism away from $p.$
Moreover, the morphism $\bar{g} : Z_m \to Z$ is proper \cite[\href{https://stacks.math.columbia.edu/tag/03GN}{Tag 03GN}]{stacks-project}.
Since $Z$ is smooth, it follows from Zariski's main theorem that $\bar{g}$ is an isomorphism.
Thus, $g : \cZ_m \to Z$ is the coarse moduli space morphism, which is an isomorphism away from $B\mu_m$ and $p$, and hence, $\cZ_m$ is a scheme in codimension $1$. 
By the bottom up characterization it follows that $\cZ_m \simeq Z$. 
Therefore, $\cZ_m$ is a scheme, and $B\mu_m \hookrightarrow \cZ_m$ has trivial stabilizer, i.e., $m = 1$.
\end{proof}

We can now give the proof of the main theorem.

\begin{proof}[Proof of Theorem \ref{main body}]
Suppose $f : \cX \to Z$ is a contraction, which contracts a smooth stacky curve $\cC \hookrightarrow \cX$ to a point $p \in \cZ$ such that $f$ restricts to an isomorphism $\cX \setminus \cC \simeq Z \setminus p.$ By Corollary \ref{contracted curve is weighted projective} and Proposition \ref{-1 curve} we see $\cC \simeq \cP(a,b)$ for some $a,b > 0$ and $\mathcal{O}_\cX(\cC) |_{\mathcal{C}} \simeq \mathcal{O}_{\cC}(-1)$. 
Thus, by Theorem \ref{existence theorem} there exists a smooth Deligne--Mumford stack $\cZ'$ and a morphism $\sigma : \cX \to \cZ'$ such that $\sigma$ is a weighted blowup of $\cZ'$ at a point.
By the universal property of a weighted blowup there is a natural morphism $\cZ' \to Z$ and arguing similar to the proof of Proposition \ref{-1 curve} using Zariski's main theorem, it follows that $\cZ' \simeq Z$ and $f : \cX \to Z$ is a weighted blowup of $\cZ$ at a point $p$ whose exceptional curve is $\cC$.
\end{proof}   

\section{An application to Hassett Moduli stack of weighted stable curves}
\label{section:applications}

In this section we apply the results obtained in the preceding sections to the reduction morphisms of Hassett moduli stacks of stable weighted curves. In \cite[Rem.~4.6]{Hassett03}, Hassett states that these reduction morphisms are smooth blowups under certain conditions on the weights; we report a proof using Theorem~\ref{main representable}. We briefly recall the definition of Hassett moduli stacks and the associated reduction morphisms. {We refer the reader to \cite{Hassett03} for further details}.

\begin{Def}\label{Def: nodal family}
For integers $g, n\geq 0$, a family of nodal curves of genus $g$ with $n$ marked points over a noetherian scheme $B$ is a flat proper morphism $\pi:C \rightarrow B$ admitting $n$ sections $s_1,s_2,\ldots,s_n$, whose geometric fibers are nodal connected curves of arithmetic genus $g$.
\end{Def}

\begin{Def}\label{Def: weighted stable nodal family}
\begin{enumerate}
    \item Define input data to be a tuple $(g,\cA):=(g,a_1,\dots,a_n)$ consisting of an integer $g\ge 0$ and the {\em weight data} $(a_1,\dots,a_n)\in \QQ^n$ satisfying $0<a_j\leq 1$ for $j=1,\ldots,n$ and $2g-2+a_1+a_2+\dots +a_n>0$.
    \item A family of nodal curves with marked points $(C,s_1,\ldots,s_n)\xrightarrow{\pi}B$ is stable of type $(g,\cA)$ if
    \begin{itemize}
        \item the sections $s_1,\ldots,s_n$ lie in the smooth locus of $\pi$,
        \item for any $\{i_1,\dots, i_r\}\subseteq [n]$ such that $\bigcap_{j=1}^{r}s_{i_j}(B)\neq\emptyset$, we have $a_{i_1}+\dots +a_{i_r}\leq 1$ and,
        \item $K_{\pi}+a_1s_1+a_2s_2+\dots +a_ns_n$ is $\pi$-ample.
    \end{itemize}
\end{enumerate}
\end{Def}

Given a choice of input data $(g,\cA)$, Hassett \cite{Hassett03} constructed the connected, smooth, proper Deligne-Mumford stack $\overline{\mathcal{M}}_{g,\cA}$ over $\ZZ$, which represents the moduli problem of pointed stable curves of type $(g,\cA)$. Furthermore, $\overline{\mathcal{M}}_{g,\cA}$ admits a projective coarse moduli space $\overline{M}_{g,\cA}$ over $\mathbb{Z}$. We refer the reader to \cite{Hassett03} for further details on construction and properties of these stacks.

Fix integers $g,n\geq 0$ and tuples $\cA=(a_1,\dots, a_n)$ and $\mathcal{B}=(b_1,\dots, b_n)$ satisfying conditions of Definition~\ref{Def: weighted stable nodal family}(1) and such that $b_j\leq a_j$ for all $1\leq j\leq n$. Then by \cite[Thm.~4.1]{Hassett03}, there exists a natural birational reduction morphism $\rho_{\mathcal{B},\cA}:\overline{\mathcal{M}}_{g,\cA}\rightarrow \overline{\mathcal{M}}_{g,\mathcal{B}}$ defined on objects by mapping $(C,s_1,\dots, s_n)\in \overline{\mathcal{M}}_{g,\cA}$ to a curve obtained by successively collapsing components of $C$ on which $K_C+b_1s_1+\dots+b_ns_n$ is not ample.

\begin{Prop}\label{Prop:reductionrepresent}
    The reduction morphisms $\rho_{\mathcal{B},\cA}:\overline{\mathcal{M}}_{g,\cA}\rightarrow \overline{\mathcal{M}}_{g,\mathcal{B}}$ are representable.
\end{Prop}

\begin{proof}
    By \cite[Lem.~4.4.3]{AV02}, it suffices to prove that the induced group homomorphism $\Aut(\xi)\rightarrow \Aut(\rho_{\mathcal{B},\cA}(\xi))$ is injective for any geometric point $\xi\in \overline{\mathcal{M}}_{g,\cA}(k)$ for any algebraically closed field $k$. Thus, using the notations from above, let $q: C\rightarrow C_{\mathcal{B}}$ be the reduction of a geometric point $\xi:=(C,s_1,\dots, s_n)\in \overline{\mathcal{M}}_{g,\cA}(k)$, and let $\alpha\in\Aut_k(C;s_1,\ldots,s_n)$ induce the identity automorphism of the $\mathcal{B}$-weighted stable curve $\rho_{\mathcal{B},\cA}(\xi):=(C_{\mathcal{B}}, q(s_1),\dots,q(s_n))$. By functoriality of the log-canonical reduction and therefore of $\rho_{\mathcal{B},\cA}$, we have $q\circ\alpha=q$. Therefore injectivity of $\Aut(\xi)\rightarrow \Aut(\rho_{\mathcal{B},\cA}(\xi))$ is equivalent to showing $\alpha=\id_C$.{ The proof that $\alpha=id_C$ is now standard, as it is the same argument used to prove that $\overline{M}_{0,n}$ is a scheme.}
    \end{proof}

For any input data $(g,\cA)$ and labeled subset $I=\{i_1,\dots, i_r\}\subset [n]$ where $n=|\cA|$, let $\cA'_I:=(a_{i_1},\dots, a_{i_r}, 1)$. Furthermore, we use the notation  $a_I:=\sum_{i\in I} a_i$, $b_I:=\sum_{i\in I}b_i$. By \cite[Prop.~4.5]{Hassett03}, the exceptional divisors of $\rho_{\mathcal{B},\cA}:\overline{\mathcal{M}}_{g,\cA}\rightarrow \overline{\mathcal{M}}_{g,\mathcal{B}}$ are exactly the boundary divisors $D_{I,J}:=\overline{\mathcal{M}}_{0, \cA'_I}\times\overline{\mathcal{M}}_{g, \cA'_J}$ where $(I,J)$ are partitions of $[n]$ indexed by
\begin{equation}\label{Equation:Indexset}
\mathfrak E_g(\cA,\mathcal{B}):=
 \left\{\begin{array}{l}(I,J): \ I:=\{i_1,\dots, i_r\}\subseteq [n],\ J:=\{j_1,\dots, j_{n-r}\}=[n]\setminus I \\
  \text{such that }|I|\geq3,\ a_I>1,
 \ b_I\leq1,\ \  2g-1+a_{J}>0.
 \end{array}
\right\}
\end{equation}
\begin{Remark}
    In particular, for every $(I,J)\in \mathfrak{E}_g(\cA, \mathcal{B})$, the divisor $D_{I,J}$ is smooth and irreducible as it is isomorphic to $\overline{\mathcal{M}}_{0, \cA'_I}\times\overline{\mathcal{M}}_{g, \cA'_J}$.
\end{Remark}

The conditions $a_I>1$ and $2g-1+a_J>0$ are required to make $(0, \cA'_I)$ and $(g, \cA'_J)$ valid input data (see Definition~\ref{Def: weighted stable nodal family}(1)). Note that $2g-1+a_J>0$ is automatic when $g\geq 1$. The conditions $b_I\leq 1$ and $|I|\geq 3$ are given by \cite[Prop.~4.5]{Hassett03}. The clutching morphism $\overline{\mathcal{M}}_{0, \cA'_I}\times\overline{\mathcal{M}}_{g, \cA'_J}\rightarrow \overline{\mathcal{M}}_{g, \cA}$ is an isomorphism onto its image {since one of the factors has arithmetic genus $0$} (a similar proof as \cite[Cor.~3.9]{Knudsen83} in the case of unweighted marked stable curves), whereby the boundary divisor $D_{I,J}$ is a smooth prime Weil divisor on a smooth Deligne-Mumford stack, and thus a smooth effective Cartier divisor (alternatively, each $D_{I,J}$ is an irreducible component of $\overline{\mathcal{M}}_{g, \cA}\setminus \mathcal{M}_{g, \cA}$, which is a stack-theoretic normal crossings divisor by \cite[Thm.~1.1]{Ulirsch15}). We now determine the exact conditions under which two exceptional prime divisors $D_{I,J}$ and $D_{I'.J'}$ intersect. This is an analogue of \cite[Fact 4]{Keel92}.

\begin{Lemma}\label{Lemma:Exceptionalintersect}
    Let $I, I'\subseteq [n]$ and $J=[n]\setminus I$ (respectively, $J'=[n]\setminus I'$) such that $(I,J),\ (I',J')\in \mathfrak E_g(\cA,\mathcal{B})$. Then $D_{I,J}\cap D_{I', J'}\neq \emptyset$ if and only if one of the following holds: i) $I\subseteq I'$ or $I'\subseteq I$ or ii) $I\cap I'=\emptyset$.
\end{Lemma}
\begin{proof}
We first show the necessity of the conditions. Suppose $D_{I,J}\cap D_{I', J'}\neq \emptyset$ and let $\xi:=(C, s_1,\dots, s_n)$ be a geometric point of the intersection. By definition of $D_{I,J}$, there is a separating node $q_I\in C$ such that one of the two connected components obtained by partially normalizing $C$ at $q_I$ is a connected genus $0$ subcurve $C_I$ carrying exactly the markings
with labels in $I$. Symmetrically, there is a separating
node $q_{I'}\in C$, partial normalization at which yields a connected genus $0$ subcurve $C_{I'}$ with markings labelled exactly in $I'$. If $q_I=q_{I'}$, then either $C_I=C_{I'}$, implying $I=I'$, or the normalization at the common node is $C_I\sqcup C_{I'}$, implying $I\cap I'=\emptyset$ (furthermore, this can happen only if $g=0$). Either case leads to conditions (i) or (ii) in the statement of the lemma.

Thus, assume that $q_I\neq q_{I'}$. Normalizing $C$ at both the nodes, yields exactly three connected components. Indeed, normalizing at $q_I$ produces two
connected components, and the distinct separating node $q_{I'}$ lies on one of them, normalizing at which further separates that component into two. Let the three
components be $C_1,C_2,C_3$ so that after regluing, $q_I$ is the node joining $C_1$ with $C_2$ and $q_{I'}$ is the node joining $C_2$ with $C_3$. For $j=1,2,3$, let $I_j:=\{i:\ s_i\in C_j\}$ and $g_j:=\rho_a(C_j)$ be the arithmetic genus of $C_j$. The sets $I_1,I_2,I_3$ are pairwise disjoint, with union $I_1\sqcup I_2\sqcup I_3=[n]$, and $g_1+g_2+g_3=g$ since both normalized nodes are separating. Thus, we have the following four possibilities of expressing $I, I'$ in terms of $I_1, I_2, I_3$:
\begin{align*}
     I=I_1,\ I'=I_3
 &\quad\Longrightarrow\quad I\cap I'=\emptyset,\\
 I=I_1,\ I'=I_1\cup I_2
 &\quad\Longrightarrow\quad I\subseteq I',\\
 I=I_2\cup I_3,\ I'=I_3
 &\quad\Longrightarrow\quad I'\subseteq I,\\
 I=I_2\cup I_3,\ I'=I_1\cup I_2
 &\quad\Longrightarrow\quad I\cup I'=[n].
\end{align*}
The first three possibilities lead to conditions (i) or (ii) in the statement of the lemma. We claim that the fourth possibility cannot hold. This is because in the fourth case, we must have $g_1+g_2=0$ and $g_2+g_3=0$ and hence $g=0$. Furthermore, $I\cup I'=[n]$ yields $b_I+b_{I'}=b_{[n]}+b_{I\cap I'}$ on the weights $(b_1,\dots, b_n)$ of the reduced curve $\rho_{\mathcal{B}, \cA}(C)$. By \eqref{Equation:Indexset}, we also have $b_I, b_{I'}\leq 1$, whereas $(g, \mathcal{B})$, with $g=0$, being a valid input data implies $b_{[n]}>2$ by Definition~\ref{Def: weighted stable nodal family}(1). We therefore obtain
\[2\geq b_I+b_{I'}=b_{[n]}+b_{I\cap I'}>2,\]
which is a contradiction. This proves the necessity of either of the conditions (i) or (ii) in the lemma. To prove sufficiency, we now assume either of these conditions and construct a geometric point of $D_{I,J}\cap D_{I',J'}$, thereby proving it to be nonempty.
\begin{enumerate}
    \item [(i)] $I\subseteq I'$ or $I'\subseteq I$: Firstly, if $I=I'$, then $D_{I,J}=D_{I',J'}$. Thus we may assume that the containments are proper. Furthermore, it suffices to only consider $I\subset I'$, as the other case is symmetric. Take three irreducible curves and join them in a chain $C:=C_1\ \cup\ C_2\ \cup\ C_3$, where $C_1\simeq\PP^1$ has $|I|$ distinct marked points labeled by $I$, the middle component $C_2\simeq\PP^1$ has $|I'|-|I|$ distinct marked points labeled by $I'\setminus I$, and $C_3$ is a smooth curve of genus $g$ with $n-|I'|$ distinct marked points labeled by $[n]\setminus I'$. Let the total $n$ marked points be $s_1,\dots, s_n$ and $q_i$ be the number of nodes on $C_i$. To check $\cA$-stability of the curve, it suffices that for $i=1,2,3$
    \begin{equation}\label{Equation:ampledegree2}
    \deg\left(\omega_C^{\otimes m}\otimes\mathcal{O}_C\left(\sum_{i=1}^{n}ma_is_i\right)\right)\Bigg|_{C_i}=2\rho_a(C_i)-2+q_i+\sum_{j:s_j\in C_i}a_j>0.
    \end{equation}
    We have $\rho_a(C_1)=\rho_a(C_2)=0$, while $\rho_a(C_3)=g$ and $q_1=q_3=1$ and $q_2=2$. Since $a_i$'s are all positive and $(I,J), (I',J')\in \mathfrak E_g(\cA,\mathcal{B})$ (where $J=[n]\setminus I$ and $J'=[n]\setminus I'$), it follows that \eqref{Equation:ampledegree2} holds for $i=1,2,3$. Thus, $(C,s_1,\dots, s_n)\in \overline{\mathcal{M}}_{g,\cA}$ and by construction it follows that it lies in $D_{I,J}\cap D_{I',J'}$.
    \item[(ii)] $I\cap I'=\emptyset$: Take two irreducible curves $C_1$ and $C_2$ isomorphic to $\PP^1$ such that $C_1$ has $|I|$ marked points labeled by $I$ and $C_2$ has $|I'|$ marked points labeled by $I'$. Take an irreducible smooth curve $C_0$ of genus $g$ with the rest of the marked points labelled by $[n]\setminus (I\cup I')$ and two distinct points $x_1,x_2\in C_0$ and join $C_1$ and $C_2$ to $C_0$ at the nodes $x_1$ and $x_2$ respectively. By similar calculations as above, one can see that the resulting curve is $\cA$-stable and is clearly a geometric point of $D_{I,J}\cap D_{I',J'}$. 
\end{enumerate}
This completes the proof of the lemma.
\end{proof}

Let $E_g(\cA, \mathcal{B}):=\{I\subseteq [n]\mid (I, [n]\setminus I)\in \mathfrak{E}_g(\cA, \mathcal{B})\}$, where $\mathfrak{E}_g(\cA, \mathcal{B})$ is defined by \eqref{Equation:Indexset}. Then $E_g(\cA, \mathcal{B})$ is a subset of the Boolean lattice of subsets of $[n]$ ordered by inclusion. Recall that an antichain in a Boolean lattice is a subset such that no two distinct elements of which can be compared by the partial order. We say that a subset of a lattice is pairwise intersecting, if any two elements of the subset have a non-zero meet (where zero is the bottom element of the lattice).

\begin{theorem}\label{Theorem:reductionblowup2}
    Let $\rho_{\mathcal{B},\cA}:\overline{\mathcal{M}}_{g,\cA}\rightarrow \overline{\mathcal{M}}_{g,\mathcal{B}}$ be the reduction morphism with respect to weight data $\cA, \mathcal{B}$ such that $E_g(\cA, \mathcal{B})$ is a pairwise intersecting antichain in the Boolean lattice of subsets of $[n]$. Then $\rho_{\mathcal{B}, \cA}$ is an ordinary blowup of $\overline{\mathcal{M}}_{g,\mathcal{B}}$ along a smooth center.
\end{theorem}

\begin{proof}
    By Lemma~\ref{Lemma:Exceptionalintersect}, if $E_g(\cA, \mathcal{B})$ is a pairwise intersecting antichain in the Boolean lattice of subsets of $[n]$, then $\cE_g(\cA, \mathcal{B}):=\bigcup_{I\in E_g(\cA, \mathcal{B})}D_{I, [n]\setminus I}$ is a disjoint union of smooth effective Cartier divisors and is hence, a smooth effective Cartier divisor itself. By \cite[Prop.~4.5]{Hassett03}, $\cE_g(\cA, \mathcal{B})$ is the exceptional divisor of the birational morphism $\rho_{\mathcal{B}, \cA}$. Thus, along with Proposition~\ref{Prop:reductionrepresent}, we obtain that $\rho_{\mathcal{B},\cA}:\overline{\mathcal{M}}_{g,\cA}\rightarrow \overline{\mathcal{M}}_{g,\mathcal{B}}$ is a representable contraction of smooth proper Deligne-Mumford stacks, whose exceptional $\cE_g(\cA, \mathcal{B})$ is a smooth effective Cartier divisor. For any $I\subseteq [n]$, let
    \[\mathcal{B}'_{[n]\setminus I}:=(b_{j_1},\dots, b_{j_{n-r}}, b_I) \text{ with } [n]\setminus I=\{j_1,\dots, j_{n-r}\} \text{ and } b_I:=\sum_{i\in I}b_i.\] 
    Then $\cY:=\bigcup_{I\in E_g(\cA, \mathcal{B})}\overline{\mathcal{M}}_{g, \mathcal{B}'_{[n]\setminus I}}$ is a smooth closed Deligne-Mumford substack of $\overline{\mathcal{M}}_{g,\mathcal{B}}$ since each $\overline{\mathcal{M}}_{g, \mathcal{B}'_{[n]\setminus I}}$ is a smooth complete intersection closed substack of codimension $|I|-1$ in $\overline{\mathcal{M}}_{g,\mathcal{B}}$ by \cite[Thm.~4.8]{AlexeevGuy_08} and for any $I\neq I'\in E_g(\cA, \mathcal{B})$, we have $\overline{\mathcal{M}}_{g, \mathcal{B}'_{[n]\setminus I}}\cap \overline{\mathcal{M}}_{g, \mathcal{B}'_{[n]\setminus I'}}=\emptyset$. By \cite[Prop.~4.5]{Hassett03}, we have $\rho_{\mathcal{B},\cA}^{-1}(\cY)=\cE_g(\cA, \mathcal{B})$ set-theoretically and $\rho_{\mathcal{B},\cA}: \overline{\mathcal{M}}_{g,\cA}\setminus \cE_g(\cA, \mathcal{B})\xrightarrow{\sim} \overline{\mathcal{M}}_{g,\mathcal{B}}\setminus \cY$ is an isomorphism. The proof is then complete by Theorem~\ref{rep theorem}.
\end{proof}

\bibliographystyle{amsalpha} 
\bibliography{ref}
\end{document}